\documentclass[11pt,reqno]{amsart}
\usepackage{amsmath}
\usepackage{amsthm}
\usepackage{graphicx}
\usepackage{amsfonts}
\usepackage{amssymb}
\usepackage{enumerate}
\usepackage{color}
\usepackage{hyperref}
\usepackage[margin=1.1in]{geometry}
\numberwithin{equation}{section}

\newtheorem{theorem}{Theorem}[section]
\newtheorem{lemma}[theorem]{Lemma}
\newtheorem{proposition}[theorem]{Proposition}

\theoremstyle{definition}
\newtheorem{example}[theorem]{Example}

\newtheorem{remark}[theorem]{Remark}
\newtheorem*{assumption*}{Assumption}

\def\E{{\mathbb E}}
\def\R{{\mathbb R}}
\def\N{{\mathbb N}}
\def\FF{{\mathbb F}}
\def\GG{{\mathbb G}}
\def\PP{{\mathbb P}}
\def\AA{{\mathbb A}}
\def\P{{\mathcal P}}
\def\F{{\mathcal F}}
\def\T{{\mathbb T}}

\def\B{{\mathcal B}}

\def\H{{\mathcal H}}
\def\X{{\mathcal X}}
\def\Y{{\mathcal Y}}
\def\Z{{\mathcal Z}}

\def\L{{\mathcal L}}
\def\G{{\mathcal G}}

\def\W{{\mathcal W}}
\def\A{{\mathcal A}}

\def\U{{\mathcal U}}

\title[Denseness of adapted processes]{Denseness of adapted processes among causal couplings}
\author{Mathias Beiglb\"{o}ck}
\author{Daniel Lacker}

\begin{document}

\begin{abstract}
It is well known that any pair of random variables $(X,Y)$ with values in Polish spaces, provided that $Y$ is nonatomic, can be approximated in joint law by random variables of the form $(X',Y)$ where $X'$ is $Y$-measurable and $X' \stackrel{d}{=} X$. This article surveys and extends some recent dynamic analogues of this result. For example, if $X$ and $Y$ are stochastic processes in discrete or continuous time, then, under a nonatomic assumption as well as a necessary and sufficient causality (or compatibility) condition, one can approximate $(X,Y)$ in law in path space by processes of the form $(X',Y)$, where $X'$ is adapted to the filtration generated by $Y$. In addition, in finite discrete time, we can take $X'$ to have the same law as $X$. A similar approximation is valid for randomized stopping times, without the first marginal fixed. Natural applications include relaxations of (mean field) stochastic control and causal optimal transport problems as well as new characterizations of the immersion property for progressively enlarged filtrations.
\end{abstract}

\maketitle

\section{Introduction}

If $X$ and $Y$ are random variables with values in Polish spaces, and if $Y$ is nonatomic, then there is a sequence $(X_n)$ of $Y$-measurable random variables such that $X_n\stackrel{d}{=} X$ and $(X_n,Y) \Rightarrow (X,Y)$, where $\stackrel{d}{=}$ denotes equality in law and $\Rightarrow$ denotes convergence in law. In other words, the set of couplings concentrated on the graph of a function (Monge couplings) is dense in the set of all couplings (Kantorovich couplings).
This well known and fundamental fact is reviewed in Section \ref{se:non-dynamic}, with a brief proof.
The goal of this paper is to survey and extend some recent \emph{dynamic} or \emph{non-anticipative} analogues of this result. 
Let us present immediately two such theorems, in discrete and continuous time:\footnote{This paper is an expansion of the the previous largely expository paper \cite{lacker2018dense}. The prior version \cite{lacker2018dense} contained only the weaker form of Theorem \ref{th:intro:density-discrete} in which condition (iii) is omitted, i.e., the $X$-marginal is not allowed to be fixed. In addition, the current version contains an expanded discussion of the continuous-time case and counterexamples in Section \ref{se:2marginals-dynamic}.}

\begin{theorem}[Discrete time] \label{th:intro:density-discrete}
Consider two stochastic processes $Y = (Y_1,\ldots,Y_N)$ and $X=(X_1,\ldots,X_N)$ with values in Polish spaces $\Y$ and $\X$, respectively. Let $\FF^Y=(\F^Y_n)_{n=1}^N$ and $\FF^X=(\F^X_n)_{n=1}^N$ denote the filtrations generated by these processes. Suppose that the law of $Y_1$ is nonatomic. Then the following are equivalent:
\begin{enumerate}[(i)]
\item $X$ is compatible with $Y$ in the sense that $\F^X_n$ is conditionally independent of $\F^Y_N$ given $\F^Y_n$, for each $n=1,\ldots,N$.
\item There exists a sequence $X^{k}=(X^{k}_1,\ldots,X^{k}_N)$ of $\FF^Y$-adapted processes such that $(Y,X^{k}) \Rightarrow (Y,X)$ in $\Y^N \times \X^N$.
\item There exists a sequence $X^{k}=(X^{k}_1,\ldots,X^{k}_N)$ of $\FF^Y$-adapted processes such that $(Y,X^{k}) \Rightarrow (Y,X)$ in $\Y^N \times \X^N$ and $X^k \stackrel{d}{=} X$.
\end{enumerate}
\end{theorem}

\begin{theorem}[Continuous time] \label{th:intro:density-continuous}
Consider two stochastic processes $Y=(Y_t)_{t \ge 0}$ and $X=(X_t)_{t \ge 0}$ with values in Polish spaces $\Y$ and $\X$, respectively, with $\X$ homeomorphic to a convex subset of a locally convex space.
Assume $X$ is continuous and $Y$ is c\`adl\`ag.
Let $\FF^Y=(\F^Y_t)_{t \ge 0}$ and $\FF^X=(\F^X_t)_{t \ge 0}$ denote the (unaugmented) filtrations generated by these processes. 
Assume either one of the following holds:
\begin{enumerate}[$\ \ $(a)]
\item $X_0$ and $Y_0$ are a.s.\ constant, and the law of $Y_t$ is nonatomic for every $t > 0$.
\item The law of $Y_0$ is nonatomic.
\end{enumerate}
Then the following are equivalent:
\begin{enumerate}[(i)]
\item $X$ is compatible with $Y$ in the sense that $\F^X_t$ is conditionally independent of $\F^Y_\infty$ given $\F^Y_t$, for each $t \ge 0$.
\item There exists a sequence $X^n$ of continuous $\FF^Y$-adapted processes such that $(X^n,Y) \Rightarrow (X,Y)$ in $C([0,\infty);\X) \times D([0,\infty);\Y)$, where $C$ and $D$ denote the continuous and Skorokhod path spaces, respectively.
\end{enumerate}
\end{theorem}

Note that in discrete time, in Theorem \ref{th:intro:density-discrete}, we may keep the $X$-marginal fixed, but we cannot do so in continuous time, at least not at the level of generality of Theorem \ref{th:intro:density-continuous}. See Section \ref{se:2marginals-dynamic} for further discussion and a counterexample, which shows that there may fail to exist \emph{any} coupling of the two given processes $X$ and $Y$ which renders $X$ adapted to the filtration of $Y$.
The proofs of Theorems \ref{th:intro:density-discrete} and \ref{th:intro:density-continuous} are given in Sections \ref{se:discretetime-twomarg} and \ref{se:dynamic-cont}, respectively. We also give two alternatives to Theorem \ref{th:intro:density-continuous} in Section \ref{se:dynamic-cont}, treating the cases where $X$ has c\`adl\`ag or merely measurable trajectories.

A measure-theoretic restatement of these results will aid in further elaboration. In the setting of either Theorem \ref{th:intro:density-discrete} or \ref{th:intro:density-continuous}, let $\nu$ and $\mu$ denote the respective laws of $Y$ and $X$ on path space; in discrete time, the path spaces are $\Y^N$ and $\X^N$, and in continuous time $D([0,\infty);\Y)$ and $C([0,\infty);\X)$. Let $\Pi^c_0(\nu)$ denote the set of joint laws of $(X,Y)$, where $X$ ranges over $\FF^Y$-adapted processes and $Y \sim \nu$, and let $\Pi^c_0(\nu,\mu)$ denote the subset for which $X \sim \mu$. Let $\Pi^c(\nu)$ denote the set of joint laws of $(X,Y)$ with $Y \sim \nu$ and satisfying the compatibility condition (i), and let $\Pi^c(\nu,\mu)$ denote the subset for which $X \sim \mu$. Note that Theorem \ref{th:intro:density-discrete} states that the closure of $\Pi^c_0(\nu,\mu)$ equals $\Pi^c(\nu,\mu)$, and thus also the closure of $\Pi^c_0(\nu)$ equals $\Pi^c(\nu)$. In continuous time, Theorem \ref{th:intro:density-continuous} shows only the latter.

We also discuss some questions of convexity and extreme points.
The set of joint laws $\Pi^c(\mu)$ is always convex. In discrete time (and in the case of randomized stopping times discussed below), $\Pi^c_0(\mu)$ is precisely the set of extreme points of $\Pi^c(\mu)$. But this fails in general in continuous time, and we give a counterexample built on a stochastic differential equation (SDE) which admits a weak solution but no strong solution. See Section \ref{se:extremepoints} for details.
Fixing both marginals, it is well known that $\Pi_0(\nu,\mu)$ is a proper subset of the set of extreme points of $\Pi(\nu,\mu)$, and it is notoriously difficult to characterize the extreme points of $\Pi(\nu,\mu)$; to enter this rabbit hole, see \cite{benes1991extremal,sudakov1979geometric}. Hence, we do not attempt to address the undoubtedly challenging problem of characterizing the extreme points of $\Pi^c(\nu,\mu)$.

Lastly, we show in Section \ref{se:randomized-stopping} that the set of \emph{randomized stopping times} is the closure of the set of true stopping times:

\begin{theorem} \label{th:intro:density-stopping}
Consider a c\`adl\`ag stochastic process $Y=(Y_t)_{t \ge 0}$ with values in a Polish space $\Y$, and let $\tau$ be a random time.
Assume $Y_t$ is nonatomic for every $t > 0$.
Let $\FF$ denote the augmented (complete and right-continuous) filtration generated by $Y$.
Then the following are equivalent:
\begin{enumerate}[(i)]
\item $\tau$ is an $\FF$-randomized stopping time in the sense that $\PP(\tau \le t \,|\, \F_\infty) = \PP(\tau \le t \,|\, \F_t)$ a.s., for every $t \ge 0$.
\item There exists a sequence of $\FF$-stopping times $\tau_n$ such that $(Y,\tau_n) \Rightarrow (Y,\tau)$ in $D([0,\infty);\Y) \times [0,\infty]$.
\end{enumerate}
\end{theorem}

Some of these theorems are not completely new, appearing in special cases and with similar proof ideas in \cite[Lemma 3.11]{carmonadelaruelacker-mfgcommonnoise} and \cite[Section 6]{carmonadelaruelacker-mfgoftiming}, written up also in the recent book \cite[Sections II.1.1.1, II.7.2.5]{carmonadelarue-book}.
The main novelty of this paper is to show in Theorem \ref{th:intro:density-discrete} that we may keep \emph{both} marginals fixed, not just the $Y$-marginal.
Additionally, and the continuous-time counterexamples of Sections \ref{se:2marginals-dynamic} and \ref{se:counterexample} are new, and we extend the results cited above to their natural levels of generality. The motivation for developing and consolidating these kinds of results in a single, concise reference stems from their relevance in increasingly diverse areas of application, some of which we discuss in the next few paragraphs.

\subsection{Optimal transport}
Optimal transport (see \cite{villanibook,villani-oldandnew,rachev-ruschendorf} for overviews) is a natural domain of application for the kinds of results discussed in this paper. Classically, letting $\Pi(\nu,\mu)$ denote the set of joint laws on $\Y \times \X$ with first marginal $\nu$ and second marginal $\mu$, the Kantorovich formulation of optimal transport is to optimize a functional of the form
\begin{align}
J(\gamma)=\int_{\Y \times \X}\Psi\,d\gamma, \qquad \gamma \in \Pi(\nu,\mu), \label{def:intro:J}
\end{align}
where $\Psi : \Y \times \X \to \R \cup \{\infty\}$ is some cost function.
This can be seen as a relaxation of the original Monge formulation, in which we optimize over the subset $\Pi_0(\nu,\mu)$ of couplings of the form $\nu(dy)\delta_{\varphi(y)}(dx)$, where $\varphi : \Y \to \X$ is measurable. When $\nu$ is nonatomic, it is well known \cite{ambrosio2003lecture,pratelli2007equality} that $\Pi_0(\nu,\mu)$ is dense in $\Pi(\nu,\mu)$, showing that the Kantorovich formulation is a genuine relaxation of the Monge formulation, at least when $J$ is continuous on $\Pi(\nu,\mu)$. 
See Sections \ref{se:2marginals} and \ref{se:optimaltransport} for discussion of this result.

Recent developments on \emph{causal optimal transport} are more in line with the focus of this paper on the dynamic setting, and we focus in this paragraph on the discrete-time setting.  Causal optimal transport, introduced by Lassalle in \cite{lassalle2013causal} and studied further in \cite{acciaio2016causal, veraguas2016causal, AcBaCa19} among others,  pertains to the optimization of functionals of the form
\begin{align}
J(\gamma) = \int_{\Y^N \times \X^N} \Psi\,d\gamma, \qquad \gamma \in \Pi^c(\nu,\mu).
\end{align}
In other words, this is an optimal transport problem with the additional constraint of compatibility.
In the causal setting, the corresponding Monge problem is to optimize $J$ only over the subset $\Pi^c_0(\nu,\mu)$ of adapted Monge couplings.
Using Theorem \ref{th:intro:density-discrete}, we can show that the Monge and Kantorovich values agree for causal optimal transport in discrete time, when the first marginal is nonatomic and the cost function sufficiently nice:

\begin{proposition} \label{pr:causal-transport-attained}
Suppose $\X$ and $\Y$ are Polish spaces and $\mu \in \P(\X^N)$, $\nu \in \P(\Y^N)$. Suppose $\Psi : \X^N \times \Y^N \rightarrow \R \cup\{\infty\}$ is of the form
\[
\Psi(y,x) = \psi(g(x),h(y)),
\]
for some Borel measurable functions $g : \X^N \to \widetilde\X$, $h : \Y^N \to \widetilde\Y$, and $\psi : \widetilde\X \times \widetilde\Y \to \R \cup\{\infty\}$, where $\widetilde\X$ and $\widetilde\Y$ are Polish spaces. Assume $\psi$  is bounded from below and lower semicontinuous in one of its variables; that is, either $\psi(\tilde x,\cdot)$ is lower semicontinuous for each $\tilde x$, or $\psi(\cdot,\tilde y)$ is lower semicontinuous for each $\tilde y$. Then the infimum $\inf_{\gamma \in \Pi^c(\nu,\mu)}\int \Psi\,d\gamma$ is attained. If, in addition it holds that
\begin{enumerate}[(i)]
\item The first marginal of $\nu\in\P(\Y^N)$ is nonatomic,
\item $\psi$ jointly continuous and finite-valued, and 
\item $|\Psi(y,x)| \le a(y)+b(x)$ for some $a \in L^1(\mu)$ and $b \in L^1(\nu)$,
\end{enumerate}
then 
\begin{align}
\inf_{\gamma \in \Pi^c(\nu,\mu)}\int \Psi\,d\gamma = \inf_{\gamma \in \Pi^c_0(\nu,\mu)}\int \Psi\,d\gamma. \label{eq:causal-monge=kantorovich}
\end{align}
\end{proposition}

The proof is given in Section \ref{se:optimaltransport}.
It remains unclear to what extent one can hope for a continuous-time analogue of \eqref{pr:causal-transport-attained}; see Section \ref{se:2marginals-dynamic} for futher discussion including a counterexample and a remarkable example due to \'Emery \cite{emery2005certain}.

\subsection{Optimal control and stopping}
In stochastic optimal control theory, one seeks to choose an adapted process in order to minimize some cost, expressed as a continuous functional of the joint law of the control and noise processes. Typically, $Y$ is an underlying noise process, and the control process $X$ must be adapted to $Y$. Framing the problem as optimization over the family $\Pi^c_0(\mu)$ of admissible joint distributions, one is naturally inclined to take the closure of $\Pi^c_0(\mu)$ to obtain a convenient topological setting, and this closure is precisely what we identified as $\Pi^c(\mu)$ above. This relaxation is now standard in stochastic control, both in discrete \cite{bertsekasshreve} and continuous time \cite{elkaroui-compactification,haussmannlepeltier-existence,kurtzstockbridge-1998}, though it is usually justified by extreme point arguments rather than denseness. In mean field contexts \cite{carmona2015forward,lacker2017limit,pham2015bellman}, however, the objective function depends nonlinearly on the law of the controlled process; extreme point arguments then break down, but one can still justify the relaxation by a denseness argument. See Section \ref{se:optimalcontrol} for details.

A similar strategy applies in optimal stopping, leading to the notion of \emph{randomized stopping time} (see \cite{baxter1977compactness,edgar1982compactness,ghoussoub1982integral}), where $\Pi^c_0(\mu)$ corresponds to true stopping times. 
This relaxation played an important role in recent analysis of American option pricing \cite{bayraktar2017super} and optimal Skorokhod embedding problems \cite{beiglboeck-cox-huesmann}.
A recent paper \cite{beiglbock2019fine} proves a denseness result in the Skorokhod embedding context which is analogous to our Theorem \ref{th:intro:density-discrete}. Namely, the set of randomized stopping times which embed a given law into a Brownian motion is dense in the set of (non-randomized) stopping times which embed the same law.

\subsection{Stochastic games}
In static (one-shot) games, it is well known that Nash equilibria typically exist among \emph{mixed strategies}, in which agents independently randomize their actions, but not necessarily among \emph{pure strategies}, in which agents choose (deterministic) actions. When stochastic factors are present, the natural analogue of a pure strategy is a measurable function from the stochastic factor to the action space, and a good notion of mixed strategy should again convexify and/or compactify the set of pure strategies. This leads, for instance, to the notion of of \emph{distributional strategy} in \cite{milgromweber-distributional}.

Similar ideas are useful in the analysis of dynamic stochastic games, where the natural analogue of a pure strategy is an adapted process. Indeed, special cases of the theorems announced above appeared first in the second named author's work on mean field games \cite{carmonadelaruelacker-mfgcommonnoise,carmonadelaruelacker-mfgoftiming}. Therein, the topology of $\Pi^c(\mu)$ is well-suited to compactness arguments, whereas optimality criteria are more easily checked using the dense subset $\Pi^c_0(\mu)$. These ideas again proved fruitful in studying the limits of $n$-player games and control problems in \cite{lacker2016general,lacker2017limit}. 

\subsection{Filtration enlargement} 
The notion of compatibility mentioned in the theorems above has appeared in a variety of contexts before, particularly in the context of enlargement of filtration, though we roughly adopt the terminology of Kurtz \cite{kurtz-yw2013}. 
The statement that $X$ is compatible with $Y$ is equivalent to the statement that every $\FF^Y$-martingale is an $\FF^{X,Y}$-martingale. We review this and other characterizations in Section \ref{se:general-compatibility}. In the literature on enlargements of filtrations, this property goes by the name \emph{H-hypothesis} or \emph{immersion}. See \cite{bremaudyor-changesoffiltrations} for its appearance in filtering, \cite{elliott2000models} for applications in credit risk, or the recent book \cite{aksamit2016enlargements}. Theorems \ref{th:intro:density-continuous} and \ref{th:intro:density-stopping} add to the pantheon of characterizations of the H-hypothesis.

\subsection{Organization of the paper}
The paper is organized as follows. Section \ref{se:non-dynamic} is something of a warm-up, collecting a number of useful facts about the set of joint distributions on $\X \times \Y$ with either one or both of the marginals fixed. Most of the results here are folklore, but complete proofs are provided. Sections \ref{se:dynamic}--\ref{se:dynamic-cont} turn to the dynamic settings in which $\X$ and $\Y$ are replaced with path spaces; this is where proofs of (more general forms of) Theorems \ref{th:intro:density-discrete} and \ref{th:intro:density-continuous} are provided.
Section \ref{se:randomized-stopping} discusses randomized stopping times and proves Theorem \ref{th:intro:density-stopping}. As mentioned briefly above, sets of joint distributions of the form of $\Pi^c(\mu)$ above are convex, and Section \ref{se:extremepoints} discusses characterizations of their extreme points, as well as a counterexample in continuous time. Lastly, Sections \ref{se:optimalcontrol} highlights an application in stochastic optimal control.

\subsection*{Notation}
Our notation throughout is as follows. We write $\R_+ = [0,\infty)$. Given a measurable space $(\Omega,\F)$, we write $\P(\Omega,\F)$ for the set of probability measures, and we abbreviate this to $\P(\Omega)$ when the $\sigma$-field is understood.
Every topological space $\X$ is equipped with its Borel $\sigma$-field, and accordingly we write $\P(\X)$ for the set of Borel probability measures on $\X$. Unless otherwise stated, we equip $\P(\X)$ with the usual topology of weak convergence, induced by bounded continuous test functions.
 The set of continuous functions from $\X$ to $\Y$ is denoted $C(\X;\Y)$. If the range space is $\R$, we write simply $C(\X)$ instead of $C(\X;\R)$.

\section{The static case} \label{se:non-dynamic}

This section serves as a warm-up, reviewing some known facts about weak convergence of joint distributions with either one or two fixed marginals for which concise proofs can be difficult to locate in the literature. 

\subsection{One fixed marginal}

First we recall an important lemma on weak convergence. It is a consequence, for instance, of \cite[Corollary 2.9]{jacodmemin-stable}, but we include a direct proof. In this section, we fix two Polish spaces $\X$ and $\Y$, as well as $\mu \in \P(\X)$. Let
\begin{align}
\Pi(\mu)  = \{P \in \P(\X \times \Y) : P(\cdot \times \Y) = \mu \} \label{def:Pi(mu)}
\end{align}
denote the set of joint laws with first marginal $\mu$.

\begin{lemma} \label{le:stable-convergence}
Suppose $P,P_n\in\Pi(\mu)$, with $P_n \rightarrow P$ weakly. Let $\varphi : \X \times \Y \rightarrow \R$ be bounded and satisfy the following:
\begin{enumerate}
\item $\varphi(\cdot,y)$ is measurable for each $y \in \Y$.
\item $\varphi(x,\cdot)$ is continuous for $\mu$-a.e. $x \in \X$.
\end{enumerate}
Then $\int\varphi\,dP_n \rightarrow \int\varphi\,dP$.
\end{lemma}
\begin{proof}
Let $c > 0$ be such that $|\varphi(x,y)| \le c$ for all $(x,y) \in \X \times \Y$.
Fix $\epsilon > 0$, and use Prokhorov's theorem to find a compact set $K \subset \Y$ such that $P_n(\X \times K) \ge 1-\epsilon$ for all $n$. Let $\X_0 \subset \X$ be a Borel set with $\mu(\X_0)=1$ such that $\varphi(x,\cdot)$ is continuous for every $x \in \X_0$. Consider the map $\Phi : \X_0 \rightarrow C(K)$ defined by $\Phi(x)(y) = \varphi(x,y)$. It is easy to prove (and follows immediately from \cite[Theorem 4.55]{aliprantisborder}) that $\Phi$ is Borel measurable.

We next show that for each $\delta > 0$ there exists a \emph{continuous} function $\varphi_\delta : \X \times K \rightarrow \R$ such that $|\varphi_\delta| \le c$ pointwise and
\[
\mu\{x \in \X : \varphi_\delta(x,y) \equiv \varphi(x,y), \ \forall y \in y\} \ge 1-\delta.
\]
First extend the domain of $\Phi$ to all of $\X$ by choosing arbitrarily some $f_0 \in C(K)$ and setting $\Phi(x) = f_0$ for $x \notin \X_0$, noting that $\Phi$ remains measurable. Next, apply Lusin's theorem \cite[Theorem 7.1.13]{bogachev-measuretheory} to find, for each $\delta > 0$, a continuous function $\Phi_\delta : \X \rightarrow C(K)$ such that $\mu\{x \in \X : \Phi_\delta(x) = \Phi(x)\} \ge 1-\delta$. Finally, define $\varphi_\delta(x,y) = (\Phi_\delta(x)(y) \wedge c) \vee (-c)$, where $c$ was the bound on $|\varphi|$.

With these preparations out of the way, the proof proceeds first with the bound
\begin{align*}
\left|\int_{\X\times \Y}\varphi\,dP_n - \int_{\X\times \Y}\varphi\,dP\right| 	&\le \left|\int_{\X\times K}\varphi\,dP_n - \int_{\X\times K}\varphi\,dP\right| + 2\epsilon c \\
	&\le \left|\int_{\X\times K}\varphi_\delta\,dP_n - \int_{\X\times K}\varphi_\delta\,dP\right| + (2\epsilon  + 4\delta)c.
\end{align*}
Because $\varphi_\delta$ is continuous on the closed set $\X \times K$, it admits a continuous extension $\bar{\varphi}_\delta$ to all of $\X \times \Y$ with $|\bar{\varphi}_\delta| \le c$ pointwise, by the Tietze extension theorem. Thus 
\begin{align*}
\left|\int_{\X\times \Y}\varphi\,dP_n - \int_{\X\times \Y}\varphi\,dP\right| 	&\le \left|\int_{\X\times \Y}\bar{\varphi}_\delta\,dP_n - \int_{\X\times \Y}\bar{\varphi}_\delta\,dP\right| + 4c(\epsilon  + \delta)
\end{align*}
Finally, continuity of $\bar{\varphi}_\delta$ and weak convergence of $P_n$ to $P$ imply 
\[
\limsup_{n\rightarrow\infty}\left|\int_{\X\times \Y}\varphi\,dP_n - \int_{\X\times \Y}\varphi\,dP\right| \le 4c(\epsilon  + \delta).
\]
As $\epsilon$ and $\delta$ were arbitrary, the proof is complete.
\end{proof}

The first part of the following proposition, mentioned in the introduction, is folklore, and we prove in Proposition \ref{pr:density-two-marginals} the more general version with both marginals fixed. See also \cite[Theorem 9.3]{ambrosio2003lecture} and \cite[Theorem 2.2.3]{youngmeasuresontopologicalspaces}.

\begin{proposition} \label{pr:density-one-marginal}
The following hold:
\begin{enumerate}[(i)]
\item If $\mu$ is nonatomic, then the following set is dense in $\Pi(\mu)$:
\begin{align}
\Pi_0(\mu) := \left\{\mu(dx)\delta_{\varphi(x)}(dy) \in \P(\X \times \Y) : \varphi : \X \rightarrow \Y \text{ is measurable} \right\}. \label{def:Pi0(mu)}
\end{align}
\item If $\mu$ is nonatomic and $\Y$ is homeomorphic to a convex subset of a locally convex space, then  the following set is dense in $\Pi(\mu)$:
\[
\left\{\mu(dx)\delta_{\varphi(x)}(dy) \in \P(\X \times \Y) : \varphi : \X \rightarrow \Y \text{ is continuous} \right\}.
\]
\end{enumerate}
\end{proposition}
\begin{proof}
The first claim will follow from Proposition \ref{pr:density-two-marginals}. To prove the second claim, we must only show that any measurable function $\varphi : \X \rightarrow \Y$ can be obtained as the $\mu$-a.s.\ limit of continuous functions. Assume without loss of generality that $\Y$ is in fact a subset of a locally convex space $\widehat\Y$, endowed with the induced topology. By a form of Lusin's theorem \cite[Theorem 7.1.13]{bogachev-measuretheory}, for each $\epsilon>0$ we may find a compact $K_\epsilon \subset \X$ such that $\mu(K_\epsilon^c) \le \epsilon$ and the restriction $\varphi|_{K_\epsilon} : K_\epsilon \rightarrow \Y$ is continuous. Using a generalization of the Tietze extension theorem due to Dugundji \cite[Theorem 4.1]{dugundji1951extension}, we may find a continuous function $\tilde{\varphi}_\epsilon : \X \rightarrow \widehat\Y$ such that $\tilde{\varphi}_\epsilon = \varphi$ on $K_\epsilon$ and such that the range $\tilde{\varphi}_\epsilon(\X)$ is contained in the convex hull of $\varphi|_{K_\epsilon}(\X)$, which is itself contained in the convex set $\Y$. We thus view $\tilde{\varphi}_\epsilon$ as a continuous function from $\X$ to $\Y$. Since $\mu(\tilde{\varphi}_\epsilon \neq \varphi) \le \mu(K_\epsilon^c) \le \epsilon$, we may find a subsequence of $\tilde{\varphi}_\epsilon$ which converges $\mu$-a.s.\ to $\varphi$.
\end{proof}

\begin{remark}
Note that part (ii) of Proposition \ref{pr:density-one-marginal} fails in general when $\Y$ fails to be (homeomorphic to) a convex set, which is most easily seen when $\Y$ is a discrete space.
\end{remark}

\begin{remark}
The two previous results can be stated in terms of \emph{stable convergence}, a mode of convergence of probability measures on product spaces studied in detail in \cite{jacodmemin-stable}. See also \cite[Section 8.10(xi)]{bogachev-measuretheory} for more references. Even when $\X$ is merely a measurable space, one may define the \emph{stable topology} on $\P(\X \times \Y)$ as the coarsest topology such that $P \mapsto \int\varphi\,dP$ is continuous for every bounded jointly measurable function $\varphi : \X \times \Y \rightarrow \R$ such that $\varphi(x,\cdot)$ is continuous for every fixed $x \in \X$. An equivalent statement of Lemma \ref{le:stable-convergence} is that, when $\X$ is a Polish space, the topologies of weak convergence and stable convergence agree on $\Pi(\mu)$.
\end{remark}

\subsection{Two fixed marginals} \label{se:2marginals}
Throughout this section, again assume $\X$ and $\Y$ are Polish spaces. We are now given marginals on both spaces, $\mu \in \P(\X)$  and $\nu \in \P(\Y)$, and we define the set of couplings:
\begin{align}
\Pi(\mu,\nu) = \{P \in \P(\X \times \Y) : P(\cdot \times \Y) =\mu, \ P(\X \times \cdot) = \nu\}. \label{def:Pi(mu,nu)}
\end{align}
The first natural question is if a two-marginal analogue of Lemma \ref{le:stable-convergence} can hold. That is, if $P,P_n \in \Pi(\mu,\nu)$ with $P_n \rightarrow P$, and if $\varphi : \X \times \Y \rightarrow \R$ is bounded and measurable, then we might expect  $\int\varphi\,dP_n$ to converge to $\int\varphi\,dP$. It turns out that this does hold for $\varphi$ of the form $\varphi(x,y)=f(x)g(y)$ for bounded measurable functions $f$ and $g$, as is shown in Lemma \ref{le:stable-convergence-two-marginals} below (see also \cite[Lemma 2.3]{beiglbock-pratelli}). For general $\varphi$, however, the subsequent Example \ref{ex:gaussian-diagonal} illustrates what can go wrong. The following will also be crucially in the proof of Theorem \ref{th:intro:density-discrete}.

\begin{lemma} \label{le:stable-convergence-two-marginals}
Suppose $P,P_n \in \Pi(\mu,\nu)$ for $n \in \N$. Then $P_n \to P$ if and only if $\int f(x)g(y)\,P_n(dx,dy) \rightarrow \int f(x)g(y)\,P(dx,dy)$ for each bounded measurable $f : \X \rightarrow \R$ and $g : \Y \rightarrow \R$.
In particular, the topology of $\Pi(\mu,\nu)$ depends only on the Borel sets of the underlying spaces $\X$ and $\Y$, not on the topologies.
\end{lemma}
\begin{proof}
Suppose first that $P_n \to P$.
By approximating $f$ uniformly by simple functions, we may assume that $f$ is itself simple. That is,
\[
f(x) = \sum_{i=1}^ma_i1_{A_i}(x),
\]
where $m \ge 1$, $a_i \in \R$, and $(A_1,\ldots,A_m)$ is a Borel partition of $\X$.
Fix $\epsilon > 0$. Let $\|\psi\|_\infty = \sup_{y \in \Y}|\psi(y)|$ for any $\psi : \Y \rightarrow \R$. By Lusin's theorem, there is a continuous function $\varphi : \Y \rightarrow \R$ such that $\|\varphi\|_\infty \le \|g\|_\infty$ and $\nu(\varphi \neq g) \le \epsilon$. Then
\begin{align*}
&\left|\int_{A_i \times \Y}g(y)P_n(dx,dy) - \int_{A_i \times \Y}g(y)P(dx,dy)\right| \\
	&\quad\le \left|\int_{A_i \times \Y}(g(y)-\varphi(y))P_n(dx,dy) - \int_{A_i \times \Y}(g(y)-\varphi(y))P(dx,dy)\right| \\
	&\quad\quad+ \left|\int_{A_i \times \Y}\varphi(y)P_n(dx,dy) - \int_{A_i \times \Y}\varphi(y)P(dx,dy)\right|.
\end{align*}
The last term tends to zero thanks to Lemma \ref{le:stable-convergence}. The other term is bounded by
\[
\|g-\varphi\|_\infty(P_n(A_i \times \{\varphi \neq g\}) + P(A_i \times \{\varphi \neq g\})) \le 4\|g\|_\infty\epsilon,
\]
because the second marginal of each $P_n$ (and of $P$) is $\nu$. This completes the proof of the ``only if" claim. For the converse, simply note that $\{P_n\}$ is tight and that every subsequential limit $Q$ must satisfy $Q(A \times B)=P(A \times B)$ for Borel sets $A \subset \X$ and $B \subset \Y$, whence $Q=P$. 
\end{proof}

\begin{example} \label{ex:gaussian-diagonal}
Suppose $\X=\Y=\R^2$ and $\mu=\nu=N(0,1)$, where $N(0,1)$ denotes the standard Gaussian law. Fix some probability space $(\Omega,\F,\PP)$ supporting a two-dimensional standard Gaussian $X$, and define $X_n$ as the rotation of $X$ by $1/n$ radians. Both $P := \PP \circ (X,X)^{-1}$ and $P_n := \PP \circ (X,X_n)^{-1}$ belong to $\Pi(\mu,\nu)$, and $P_n \rightarrow P$. But if $\varphi(x,y) = 1_{\{x=y\}}$, then $\int\varphi\,dP_n = 0$ for all $n$ while $\int\varphi\,dP=1$.
\end{example}

Let $\Pi_0(\mu,\nu)$ denote the set of \emph{Monge} couplings, defined as 
\begin{align}
\Pi_0(\mu,\nu) = \left\{\mu(dx)\delta_{\varphi(x)}(dy) \in \Pi(\mu,\nu) : \varphi : \X \rightarrow \Y \text{ is measurable} \right\}. \label{def:Pi0(mu,nu)}
\end{align}
That is, a coupling of $(\mu,\nu)$ belongs to $\Pi_0(\mu,\nu)$ if and only if it is concentrated on the graph of a measurable function. The following well known proposition, implicit in \cite{pratelli2007equality}, shows that if $\mu$ is nonatomic then $\Pi_0(\mu,\nu)$ is dense in $\Pi(\mu,\nu)$, which gives Proposition \ref{pr:density-one-marginal} as a corollary.
Borel isomorphism is the only non-elementary ingredient in the proof.

\begin{proposition} \label{pr:density-two-marginals}
If $\mu$ is nonatomic, then $\Pi_0(\mu,\nu)$ is dense in $\Pi(\mu,\nu)$.
\end{proposition}
\begin{proof}
Fix $P \in \Pi(\mu,\nu)$. Let $\pi_n=\{A_{n,i} : i=1,\ldots,n\}$ be a sequence of partitions of $\X$ such that $\pi_n \subset \pi_{n+1}$ and $\cup_n\pi_n$ generates the Borel $\sigma$-field. Let $\eta_{|A} = \eta(A \cap \cdot)$ denote the trace of a measure $\eta$ on a set $A$. 

Define the finite measure $\widetilde{\nu}^n_i(\cdot) = P(A_{n,i} \times \cdot)$, which has total mass $\widetilde{\nu}^n_i(\Y) = P(A_{n,i} \times \Y) = \mu(A_{n,i})$. Because $\mu$ is nonatomic, there exists by Borel isomorphism\footnote{See \cite[Theorem 15.6]{kechris} for Borel isomorphism, and note that a Borel subset of a Polish space is always a standard Borel space \cite[Lemma 13.4]{kechris}.} a measurable function $\varphi_{n,i} : A_i \rightarrow \Y$ such that $\mu_{|A_{n,i}} \circ \varphi_{n,i}^{-1} = \widetilde{\nu}_i$. Now define $\varphi_n : \X \rightarrow \Y$ by setting $\varphi_n = \varphi_{n,i}$ on $A_{n,i}$ for each $i$.  We claim that $\mu \circ \varphi_n^{-1} = \nu$. Indeed, for a Borel set $B \subset \Y$,
\begin{align*}
\mu \circ \varphi_n^{-1}(B) &= \sum_{i=1}^n\mu(A_{n,i} \cap \varphi^{-1}(B)) = \sum_{i=1}^n\mu_{|A_{n,i}}(\varphi_{n,i}^{-1}(B)) = \sum_{i=1}^n\widetilde{\nu}^n_i(B) \\
	&= \sum_{i=1}^nP(A_{n,i} \times B) = P(\X \times B) = \nu(B).
\end{align*}
Now define $P_n(dx,dy) = \mu(dx)\delta_{\varphi_n(x)}(dy)$. By construction, we have $P_n \in \Pi(\mu,\nu)$ and
\begin{align*}
\int_{A_{n,i} \times \Y}f(y)\,P_n(dx,dy) = \int_Ff\,d\widetilde{\nu}^n_i = \int_{A_{n,i} \times \Y}f(y)\,P(dx,dy)
\end{align*}
for every $i=1,\ldots,n$ and every bounded measurable $f : \Y \rightarrow \R$. Hence, $P_n = P$ on $\sigma(\pi_n) \otimes \B(\Y)$, where $\B(\Y)$ denotes the Borel $\sigma$-field of $\Y$. Because $\pi_n \subset \pi_{n+1}$ for each $n$ and $\sigma(\cup_n\pi_n)=\B(\X)$, it is straightforward to conclude that $P_n \rightarrow P$.
\end{proof}

\begin{remark} \label{re:bimeasurable}
In Proposition \ref{pr:density-two-marginals}, when both marginals are nonatomic, the result can be refined so that the approximations are \emph{bimeasurable}. Precisely, suppose we are given a pair of nonatomic random variables $X$ and $Y$ with values in Polish spaces $\X$ and $\Y$. Then there exists a sequence $(X_n,Y_n)$ of $\X \times \Y$-valued random variables such that $X_n$ is $Y_n$-measurable, $Y_n$ is $X_n$-measurable, $X_n \stackrel{d}{=} X$, $Y_n \stackrel{d}{=} Y$, and $(X_n,Y_n) \Rightarrow (X,Y)$. See \cite[Proposition A.3]{gangbo1999monge} or \cite[pp. 296]{emery2005certain} for proofs in the case where $X$ and $Y$ are uniform in $[0,1]$, which extends to general spaces by Borel isomorphism and Lemma \ref{le:stable-convergence-two-marginals}.
\end{remark}

\begin{remark}
Results closely related to Proposition \ref{pr:density-two-marginals} are well known in the theory of copulas, mainly for $\X=\Y=[0,1]$ and when $\mu=\nu$ are both the uniform (Lebesgue) measure. In this case, a dense subset of $\Pi(\mu,\nu)$ is given by the set of Monge couplings induced by piecewise linear bijections with slope 1, known as \emph{shuffles of min} \cite{mikusinski1992shuffles}.
\end{remark}

\begin{remark}
Combining Proposition \ref{pr:density-two-marginals} with a remarkable result of Oxtoby \cite[Theorem 1]{oxtoby1973approximation} yields the following: Suppose $\X$ is a connected $n$-dimensional manifold for some $n \ge 2$, and suppose $\mu \in \P(\X)$ is nonatomic and charges every nonempty open set. Then a dense subset of $\Pi(\mu,\mu)$ is given by the set of measures of the form $\mu(dx)\delta_{\varphi(x)}(dy)$, where $\varphi : \X \rightarrow \X$ is a homeomorphism with $\mu \circ \varphi^{-1} = \mu$.
\end{remark}

\subsection{Optimal transport} \label{se:optimaltransport}

We state here two natural applications of the results of the previous section, both contained in the following Proposition, which are modest extensions of known results.
The first deals with attainment of the infimum in optimal transport problems, while the second shows that the Kantorovich transport problem is often a genuine relaxation of the Monge problem, in the sense that the infima agree.  A different version of the latter result was shown in \cite[Theorem B]{pratelli2007equality}, generalizing \cite[Theorem 2.1]{ambrosio2003lecture}. Our continuity requirements are weaker than in \cite[Theorem B]{pratelli2007equality}, but our integrability requirements are stronger.

\begin{proposition} \label{pr:transport-attained}
Suppose $\X$ and $\Y$ are Polish spaces and $\mu \in \P(\X)$, $\nu \in \P(\Y)$. Suppose $\Psi : \X \times \Y \rightarrow \R \cup\{\infty\}$ is of the form
\[
\Psi(x,y) = \psi(g(x),h(y)),
\]
for some Borel measurable functions $g : \X \to \widetilde\X$, $h : \Y \to \widetilde\Y$, and $\psi : \widetilde\X \times \widetilde\Y \to \R \cup\{\infty\}$, where $\widetilde\X$ and $\widetilde\Y$ are Polish spaces. Assume $\psi$ is bounded from below and lower semicontinuous in one of its variables; that is, either $\psi(x,\cdot)$ is lower semicontinuous for each $x$ or $\psi(\cdot,y)$ is lower semicontinuous for each $y$. Then the infimum $\inf_{P \in \Pi(\mu,\nu)}\int \Psi\,dP$ is attained. If, in addition it holds that
\begin{enumerate}[(i)]
\item $\mu$ is nonatomic,
\item $\psi$ jointly continuous and finite-valued, and 
\item $|\Psi(x,y)| \le a(x)+b(y)$ for some $a \in L^1(\mu)$ and $b \in L^1(\nu)$,
\end{enumerate}
then 
\begin{align}
\inf_{P \in \Pi(\mu,\nu)}\int \Psi\,dP = \inf_{P \in \Pi_0(\mu,\nu)}\int \Psi\,dP. \label{eq:monge=kantorovich}
\end{align}
\end{proposition}
\begin{proof}
The map $\Pi(\mu,\nu) \to \Pi(\mu\circ g^{-1},\nu \circ h^{-1})$ given by $P \mapsto P \circ (g,h)^{-1}$ is continuous by Lemma \ref{le:stable-convergence-two-marginals}, where we understand $(g,h)$ to mean the function $(x,y) \mapsto (g(x),h(y))$. Moreover, the map $\Pi(\mu\circ g^{-1},\nu \circ h^{-1}) \ni P \mapsto \int \psi\,dP$ is lower semicontinuous by (a simple extension of) Lemma \ref{le:stable-convergence}.
Since also $\Pi(\mu,\nu)$ is weakly compact and, we deduce the first claim, that the infimum $\inf_{P \in \Pi(\mu,\nu)}\int \Psi\,dP$ is attained.

To prove the second claim, let $P \in \Pi(\mu,\nu)$. By Proposition \ref{pr:density-two-marginals}, we can find $P_n \in \Pi_0(\mu,\nu)$ with $P_n \to P$. 
As above, Lemma \ref{le:stable-convergence-two-marginals} and continuity of $\psi$ ensure that the map $P \mapsto P \circ \Psi^{-1}$ is continuous from $\Pi(\mu,\nu)$ to $\P(\R)$. The assumption that $|\Psi(x,y)| \le a(x)+b(y)$  with $\int a\,d\mu$ and $\int b\,d\nu$ finite ensures the uniform integrability
\begin{align*}
\lim_{r\to\infty}\sup_n \int_{\{|\Psi| \ge r\}} \psi \, dP_n = 0.
\end{align*}
It follows easily (e.g., by \cite[Lemma 4.11]{kallenberg-foundations}) that $\int \psi\,dP_n \to \int \psi\,dP$, and we deduce \eqref{eq:monge=kantorovich}.
\end{proof}

\begin{proof}[Proof of Proposition \ref{pr:causal-transport-attained}]
The proof of Proposition \ref{pr:causal-transport-attained} is identical to that of Proposition \ref{pr:transport-attained}, except using Theorem \ref{th:intro:density-discrete} in place of Proposition \ref{pr:density-two-marginals}.
\end{proof}

\section{Discrete time, with one fixed marginal} \label{se:dynamic}

We now begin extending the results of the previous section to the dynamic setting. We warm up with a result which is, for the most part, weaker than Theorem \ref{th:intro:density-discrete}, in the sense that the $X$-marginal cannot be fixed and $\X$ is required to be convex, but as a tradeoff we can take the approximations $X^{(k)}$ to be \emph{continuous} functions of $Y$. In addition, the proof is much shorter and more transparent, following a natural induction; convexity of $\X$ is used in order to obtain \emph{continuous} approximations at each time step (using Proposition \ref{pr:density-one-marginal}), and this continuity crucially allows each step in the induction to respect the approximation from the previous step. The first appearance of a result of this nature seems to be \cite[Lemma 3.11]{carmonadelaruelacker-mfgcommonnoise}, which is in continuous time but implicitly contains the discrete time result, which is also proven in more generality in \cite[Proposition 6.2]{carmonadelaruelacker-mfgoftiming}.

In the following, for a discrete-time stochastic process $X=(X_1,\ldots,X_N)$ we write $\FF^X=(\F^X_n)_{n=1,\ldots,N}$ for the corresponding natural filtration, namely $\F^X_n = \sigma(X_1,\ldots,X_n)$. Recall that another process $Y=(Y_1,\ldots,Y_n)$ is said to be \emph{adapted to $\FF^X$}, or simply \emph{$X$-adapted}, if $Y_n$ is $\F^X_n$-measurable for every $n$.

\begin{theorem} \label{th:density-discrete}
Consider two stochastic processes $Y = (Y_1,\ldots,Y_N)$ and $X=(X_1,\ldots,X_N)$ with values in Polish spaces $\Y$ and $\X$, respectively, where $\X$ is homeomorphic to a convex subset of a locally convex space. Suppose that the law of $Y_1$ is nonatomic. Then the following are equivalent:
\begin{enumerate}[(i)]
\item For each $n$, $\F^X_n$ is conditionally independent of $\F^Y_N$ given $\F^Y_n$.
\item Then there exists a sequence $X^{(k)}=(X^{(k)}_1,\ldots,X^{(k)}_N)$ of $(\F^Y_n)_{n=1}^N$-adapted processes such that $(Y,X^{(k)}) \Rightarrow (Y,X)$ in $\Y^N \times \X^N$. 
\item Then there exists a sequence $X^{(k)}=(X^{(k)}_1,\ldots,X^{(k)}_N)$ of $(\F^Y_n)_{n=1}^N$-adapted processes such that $(Y,X^{(k)}) \Rightarrow (Y,X)$ in $\Y^N \times \X^N$, and $X^{(k)}$ is of the form $X^{(k)}_n = f^k_n(Y_1,\ldots,Y_n)$ for some \emph{continuous} functions $f^k_n : \Y^n \rightarrow \X$.
\end{enumerate}
\end{theorem}
\begin{proof}
We begin by proving that (i) implies (ii), as this is the more challenging step.
The proof is an inductive application of Proposition \ref{pr:density-one-marginal}(ii). First, use Proposition \ref{pr:density-one-marginal} to find a sequence of continuous functions $h^j_1 : \Y \rightarrow \X$ such that $(Y_1,h^j_1(Y_1)) \Rightarrow (Y_1,X_1)$ as $j\rightarrow\infty$. Let us show that in fact $(Y,h^j_1(Y_1)) \Rightarrow (Y,X_1)$.
Let $\varphi : \Y^N \rightarrow \R$ be bounded and measurable, and let $\psi : \X \rightarrow \R$ be continuous. Use Lemma \ref{le:stable-convergence} as well as the conditional independence of $Y$ and $X_1$ given $Y_1$ to get
\begin{align*}
\lim_{j\rightarrow\infty}\E[\varphi(Y)\psi(h^j_1(Y_1))] &= \lim_{j\rightarrow\infty}\E\left[\E\left[\left.\varphi(Y)\right| Y_1\right]\psi(h^j_1(Y_1))\right]	\\
	&= \E\left[\E\left[\left.\varphi(Y)\right| Y_1\right]\psi(X_1)\right] \\
	&= \E\left[\E\left[\left.\varphi(Y)\right| Y_1,X_1\right]\psi(X_1)\right] \\
	&= \E\left[\varphi(Y)\psi(X_1)\right]
\end{align*}
This is enough to show that $(Y,h^j_1(Y_1)) \Rightarrow (Y,X_1)$ (see, e.g., \cite[Proposition 3.4.6(b)]{ethier-kurtz}).

We proceed inductively as follows. Abbreviate $Y^n := (Y_1,\ldots,Y_n)$ and $X^n := (X_1,\ldots,X_n)$ for each $n=1,\ldots,N$, noting $Y^N=Y$. Suppose we are given $1 \le n < N$ and continuous functions  $g^j_k : \Y^k \rightarrow \X$, for $k \in \{1,\ldots,n\}$ and $j \ge 1$, satisfying
\begin{align}
\lim_{j\rightarrow\infty}(Y,g^j_1(Y^1),\ldots,g^j_n(Y^n)) = (Y,X_1,\ldots,X_n), \label{pf:density1}
\end{align}
where the convergence here and throughout the proof is in law.
We will show that there exist continuous functions $h^i_k : \Y^k \rightarrow \X$ for each $k \in \{1,\ldots,n+1\}$ and $i \ge 1$ such that 
\begin{align}
\lim_{i\rightarrow\infty}(Y,h^i_1(Y^1),\ldots,h^i_{n+1}(Y^{n+1})) = (Y,X_1,\ldots,X_{n+1}). \label{pf:density2}
\end{align}
By Proposition \ref{pr:density-one-marginal} there exists a sequence of continuous functions $\widehat{g}^j : (\Y^{n+1} \times \X^n) \rightarrow \X$ such that
\begin{align*}
\lim_{j\rightarrow\infty}(Y^{n+1},X_1,\ldots,X_n,\widehat{g}^j(Y^{n+1},X^n)) = (Y^{n+1},X_1,\ldots,X_n,X_{n+1}).
\end{align*}
Note that $Y$ and $(Y^n,X^n)$ are conditionally independent given $Y_1$. Using the same argument as above, it follows that in fact
\begin{align}
\lim_{j\rightarrow\infty}(Y,X_1,\ldots,X_n,\widehat{g}^j(Y^{n+1},X^n)) = (Y,X_1,\ldots,X_n,X_{n+1}). \label{pf:density3}
\end{align}
By continuity of $\widehat{g}^j$, the limit \eqref{pf:density1} implies (using Lemma \ref{le:stable-convergence} and the fact that $Y^{n+1}$ is $Y$-measurable) that, for each $j$,
\begin{align}
\lim_{i\rightarrow\infty}&(Y,g^i_1(Y^1),\ldots,g^i_n(Y^n),\widehat{g}^j(Y^{n+1},g^i_1(Y^1),\ldots,g^i_n(Y^n))) \nonumber \\
	&= (Y,X_1,\ldots,X_n,\widehat{g}^j(Y^{n+1},X_1,\ldots,X_n)). \label{pf:density4}
\end{align}
Combining the two limits \eqref{pf:density3} and \eqref{pf:density4}, we may find a subsequence $j_i$ such that 
\begin{align*}
\lim_{i\rightarrow\infty}&(Y,g^{j_i}_1(Y^1),\ldots,g^{j_i}_n(Y^n),\widehat{g}^i(Y^{n+1},g^{j_i}_1(Y^1),\ldots,g^{j_i}_n(Y^n))) \\
	&= (Y,X_1,\ldots,X_n,X_{n+1}).
\end{align*}
Define $h^i_k := h^{j_i}_k$ for $k=1,\ldots,n$ and $h^i_{n+1}(Y^{n+1}) := \widehat{g}^i(Y^{n+1},g^{j_i}_1(Y^1),\ldots,g^{j_i}_n(Y^n))$ 
to establish \eqref{pf:density2}. This completes the proof that (i) implies (iii).

Clearly (iii) implies (ii).
To prove (ii) implies (i), note first that (i) holds if and only if for every $n$ and every bounded measurable functions $f_n$, $h_n$, and $g$ on $\X^n$, $\Y^n$, and $\Y^N$, respectively, we have 
\begin{align}
\E[f_n(X^n)h_n(Y^n)(g(Y) - g_n(Y^n))] &= 0, \label{pf:compat-constraint} \\
\text{ where } g_n(Y^n) &:= \E[g(Y) | \F^Y_n]. \nonumber
\end{align}
In fact, by a routine approximation, it is enough that this holds for \emph{continuous} $f_n$. Hence, if a sequence of processes $X^{(k)}=(X^{(k)}_1,\ldots,X^{(k)}_N)$ satisfies \eqref{pf:compat-constraint} as well as $(X^{(k)},Y) \Rightarrow (X,Y)$ for some process $X$, then using Lemma \ref{le:stable-convergence} we conclude that $X$ must satisfy \eqref{pf:compat-constraint} as well. To conclude that (ii) implies (i), simply note that a $Y$-adapted process $X$ trivially verifies \eqref{pf:compat-constraint}.
\end{proof}

\section{Discrete time, with two fixed marginals} \label{se:discretetime-twomarg}

In this section we prove Theorem \ref{th:intro:density-discrete}. We find it convenient to work with a more measure-theoretic notation. Let $\Y_0,\ldots,\Y_N$ and $\X_0,\ldots,\X_N$ be Polish spaces, and let $\Y=\Y_0 \times \cdots \times \Y_N$ and $\X = \X_0 \times \cdots \times \X_N$. For probability measures $\nu \in \P(\Y)$ and $\mu \in \P(\X)$, we define the sets $\Pi^c(\nu,\mu)$ and $\Pi^c_0(\nu,\mu)$ of  causal couplings of Kantorovich and Monge type. Precisely, $\Pi^c(\nu,\mu)$ is the set of laws of $\Y \times \X$-valued random variables $(Y,X)$, where $Y=(Y_0,\ldots,Y_N) \sim \nu$, $X=(X_0,\ldots,X_N) \sim \mu$, and $(X_0,\ldots,X_n)$ is conditionally independent of $Y$ given $(Y_0,\ldots,Y_n)$, for each $n=1,\ldots,N$. 
Similarly, $\Pi^c_0(\nu,\mu)$ is the set of laws of $\Y \times \X$-valued random variables $(Y,X)$, where $Y=(Y_0,\ldots,Y_N) \sim \nu$, $X=(X_0,\ldots,X_N) \sim \mu$, and $X$ is adapted to the filtration generated by $Y$.

Note that we allow here that the spaces $\Y_0,\ldots,\Y_N$ and $\X_0,\ldots,\X_N$ be distinct, which may appear somewhat more general than Theorem \ref{th:intro:density-discrete}, but it is in fact equivalent, as can be argued via disjoint unions. In addition, we number here from $0$ to $N$ rather than $1$ to $N$, for no reason other than convenience.

We have already seen in Theorem \ref{th:density-discrete} that $\Pi^c(\nu,\mu)$ is a closed set, and we must only prove that $\Pi^c_0(\nu,\mu)$ is dense in $\Pi^c(\nu,\mu)$. We accomplish this via the following strategy:
\begin{enumerate}
\item First, we prove the theorem under the additional hypothesis that $Y_0$ is independent of $(Y_1,\ldots,Y_N)$, i.e., $\nu$ is of the form 
\begin{align}
\nu = \nu_0 \times \nu_{+}, \qquad \text{for some } \nu_0 \in \P(\Y_0) \text{ and } \nu_{+} \in \P(\Y_1 \times \cdots \times \Y_N). \label{productform}
\end{align}
\item Second, we prove the case where $\nu$ is a convex combination of finitely many measures of the form \eqref{productform}.
\item Assuming next that $\Y_1 = \cdots = \Y_N = \{1,\ldots,m\}$ for some $m \in \N$, we prove the claim without any additional structural hypotheses.
\item Finally, we remove the restriction that $\Y_1,\ldots,\Y_N$ are finite, arriving at the fully general case.
\end{enumerate}

\subsection{Step 1: The independent case}

We begin with two preparatory lemmas which will illustrate the utility of the assumption that $Y_0$ is independent of $(Y_1,\ldots,Y_N)$.
Throughout this section, we write $\lambda$ for Lebesgue measure on $[0,1]$. For a measure $\mu$ and a measurable set $A$ we write $\mu_{|A}$ for the trace on $A$, defined by $\mu_{|A}(\cdot) := \mu(A \cap \cdot)$.

\begin{lemma} \label{le:random-generator}
Let $X_0$ and $Y_0$ be random variables taking values in Polish spaces $\X_0$ and $\Y_0=[0,1]$, respectively. Assume $Y_0 \sim \lambda$. For each $n \in \N$, there exist measurable functions $T=(T_1,T_2) : \Y_0 \to \Y_0 \times \X_0$ and $u : \Y_0 \to [0,1]$ such that:
\begin{enumerate}
\item $u(Y_0)$ is independent of $T(Y_0)$.
\item $u(Y_0) \sim \lambda$.
\item $T(Y_0) \stackrel{d}{=} (Y_0,X_0)$.
\item $|y-T_1(y)| \le 1/n$ for all $y \in \Y_0=[0,1]$.
\end{enumerate}
\end{lemma}
\begin{proof}
Fix $n \in \N$, let $\lambda$ and $\gamma$ denote the laws of $Y_0$ and $(Y_0,X_0)$, respectively. For positive integers $k < n$ let $I_k = [(k-1)/n,k/n)$, and set $I_n = [(n-1)/n,1]$. Since $\lambda$ is nonatomic, for each $k=1,\ldots,n$ there exists a measurable bijection $\varphi_k : I_k \to I_k \times \X_0 \times [0,1]$ such that $\lambda_{|I_k} \circ \varphi_k^{-1} = \gamma_{|I_k \times \X_0} \times \lambda$. Define $\psi : \Y_0 \to \Y_0 \times \X_0 \times [0,1]$ by setting $\psi(y) := \varphi_k(y)$ for $y \in I_k$. Then
\begin{align*}
\lambda \circ \psi^{-1} = \sum_{k=1}^n\gamma_{|I_k \times \X_0} \times \lambda = \gamma \times \lambda.
\end{align*}
Moreover, if writing $\psi(y)=(T_1(y),T_2(y),u(y))$, we see that $|T_1(y)-y| \le 1/n$ for all $y \in [0,1]$.
\end{proof}

We next need a somewhat non-standard dynamic version of a \emph{transfer principle}. We omit the proof, as it is nearly identical to that of Lemma \ref{le:transfer-dynamic} given in Section \ref{se:transfer} below.

\begin{lemma} \label{le:transfer-dynamic-special}
Suppose $X=(X_n)_{n=0}^N$ and $Y=(Y_n)_{n=0}^N$ are stochastic processes with values in Polish spaces $\X$ and $\Y$. Suppose $X_n$ is conditionally independent of $\F^Y_N$ given $\F^Y_n$, for every $n\geq 0$, then there exist measurable functions $f_n : \X \times \Y^n \times [0,1] \rightarrow \X$ and (perhaps on an extension of the probability space) an independent uniform random variable $U$ such that, if $Z = (Z_0, Z_1,\ldots,Z_N)$ where
\begin{align}\label{eq:transfer-def}
Z_0=X_0,\quad Z_n = f_n(X_0,Y_0, Y_1,\ldots,Y_n,U), \quad \text{ for } n =1,\ldots,N,
\end{align}
then $(Z,Y)  \stackrel{d}{=} (X,Y)$.
\end{lemma}

We are now ready to state and prove a special case of Theorem \ref{th:intro:density-discrete}, which will aid in our proof of the fully general version. In fact, under the additional assumption here that $Y_0$ is independent of $(Y_1,\ldots,Y_N)$, we are able to construct the desired approximating processes $X^k$ along with an additional  $Y_0$-measurable uniform random variables independent of $X^k$, which will be useful later on.

\begin{proposition} \label{pr:density-discrete-specialcase1}
Consider two random variables $Y = (Y_0,\ldots,Y_N)$ and $X=(X_0,\ldots,X_N)$ with values in the Polish spaces $\Y=\Y_0\times\cdots\times\Y_n$ and $\X=\X_0\times\cdots\times\X_n$, respectively. Let $\FF^Y=(\F^Y_n)_{n=0}^N$ and $\FF^X=(\F^X_n)_{n=0}^N$ denote the filtrations generated by these processes. Suppose that $Y_0$ is independent of $(Y_1,\ldots,Y_N)$ and that the law of $Y_0$ is nonatomic. 
Finally, suppose  $X$ is compatible with $Y$ in the sense that $\F^X_n$ is conditionally independent of $\F^Y_N$ given $\F^Y_n$, for each $n=0,\ldots,N$.
Then there exist $\FF^Y$-adapted processes $X^k=(X^k_0,X^{k}_1,\ldots,X^{k}_N)$ and measurable functions $u^k : \Y_0 \to [0,1]$ such that $(Y,X^{k}) \Rightarrow (Y,X)$ in $\Y \times \X$, $X\stackrel{d}{=}X^k$ for each $k$, and $u^k(Y_0) \sim \lambda$ is independent of $X^k$ for each $k$.
\end{proposition}
\begin{proof}
By Borel isomorphism and Lemma \ref{le:stable-convergence-two-marginals}, we may assume that $\Y=[0,1]$ and $Y_0 \sim \lambda$. 
For each $k\in \N$, construct functions $T^k=(T^k_1,T^k_2) : \Y_0 \to \Y_0 \times \X_0$, $u^k : \Y_0 \to [0,1]$ as in Lemma \ref{le:random-generator}. Construct an independent uniform random variable $U$ and define $Z=(Z_0,\ldots,Z_N)$ and $Z^k=(Z^k_0,\ldots,Z^k_N)$ via
\begin{align*}
Z_0&=X_0,\quad &  &Z_n= f_n(X_0,Y_0, Y_1,\ldots,Y_n,U)  \\
Z^k_0&=T^k_2(Y_0),\quad & &Z_n^k = f_n(T^k_2(Y_0),T^k_1(Y_0), Y_1,\ldots,Y_n,u^k(Y_0)), \quad n=1,\ldots,N.
\end{align*}
That is, the process $Z$ is defined exactly as in Lemma \ref{le:transfer-dynamic} with $(Z,Y) \stackrel{d}{=} (X,Y)$, whereas $Z'$  mimics its behavior based on the functions defined in Lemma \ref{le:random-generator}. 
Indeed, since $X_0$ is conditionally independent of $Y_1, \ldots, Y_n$ given $Y_0$, and since $Y_0$ is assumed to be independent of $Y_1, \ldots, Y_n$, the pair $(X_0, Y_0)$ is independent of $Y_1, \ldots, Y_n$. We thus obtain 
\[
(X_0, Y_0, Y_1, \ldots, Y_n,U)\stackrel{d}{=}
(T^k_2(Y_0), T^k_1(Y_0), Y_1, \ldots, Y_n, u^k(Y_0)),
\]
which implies $Z^k \stackrel{d}{=} Z$, as well as
\[
(Z^k, T^k_1(Y_0), Y_1, \ldots, Y_n) \stackrel{d}{=} (Z,Y) \stackrel{d}{=} (X,Y).
\]
Finally, since $|y-T^k_1(y)| \le 1/k$, we have $(Z^k, Y) \Rightarrow (X,Y)$, completing the proof.
\end{proof}

\subsection{Step 2: Finite combinations of the independent case}

\begin{proposition} \label{pr:density-discrete-specialcase2}
Adopt the setting and notation of Proposition \ref{pr:density-discrete-specialcase1}, but instead of assuming that $Y_0$ is independent of $(Y_1,\ldots,Y_N)$, we assume that the law $\nu$ of $(Y_0,\ldots,Y_N)$ is of the form
\begin{align*}
\nu = \sum_{i=1}^k a_i\,\nu^i_0 \times \nu^i_+,
\end{align*}
where $k \in \N$, $\nu^1_0,\ldots,\nu^k_0 \in \P(\Y_0)$ have disjoint supports, $\nu^i_+,\ldots,\nu^k_+ \in \P(\Y_1\times\cdots\times\Y_N)$, and $a_i \ge 0$ with $\sum_{i=1}^ka_i=1$.
Then the conclusion of Proposition \ref{pr:density-discrete-specialcase1} remains valid.
\end{proposition}
\begin{proof}
We simply apply Proposition \ref{pr:density-discrete-specialcase1} for each $i=1,\ldots,k$, and then assemble the resulting approximations (on the disjoint supports of $\nu^i_0$) into a single one.
\end{proof}

\subsection{Step 3: Finite state space, without independence}

\begin{proposition} \label{pr:density-discrete-specialcase3}
Adopt the setting and notation of Proposition \ref{pr:density-discrete-specialcase1}, but instead of assuming that $Y_0$ is independent of $(Y_1,\ldots,Y_N)$, we assume that 
\begin{align*}
\Y_1=\Y_2=\ldots=\Y_N = \{1,\ldots,m\}^N, \quad m \in \N,
\end{align*}
and also that $\Y_0=[0,1]$ with $Y_0 \sim \lambda$.
Then $\Pi^c_0(\nu,\mu)$ is dense in $\Pi^c(\nu,\mu)$.
\end{proposition}
\begin{proof}
Let $\nu$ and $\mu$ denote the laws of $Y=(Y_0,\ldots,Y_N)$ and $X=(X_0,\ldots,X_N)$, respectively, and let $\gamma$ denote their joint law. The idea is to approximate $\nu$ (and thus $\gamma$ as well) by a measure satisfying the structural condition of Proposition \ref{pr:density-discrete-specialcase2}. This approximation alters the desired marginals, but we use the independent uniforms $u^k(Y_0)$ to fix this.

Let $\eta > 1$, which we will later send to $1$. Noting that $\P(\{1,\ldots,m\})$ is compact, we will essentially approximate this space by a finite $(\eta-1)$-net, but we must be careful about the precise nature of the approximation. We may disintegrate $\mu$ and $\nu$ in the form
\begin{align*}
\nu(dy_0,\ldots,dy_N) &= \lambda(dy_0)\nu_1^{y_0}(dy_1)\cdots\nu_N^{y_0\ldots y_{N-1}}(dy_N), \\
\mu(dy_0,\ldots,dy_N) &= \mu_0(dx_0)\mu_1^{y_0}(dy_1)\cdots\mu_N^{y_0\ldots y_{N-1}}(dy_N).
\end{align*}
Moreover, since $\Y_1=\cdots=\Y_N=\{1,\ldots,m\}$, we can find measurable functions $a_t : \Y_0 \times \{1,\ldots,m\}^t \to [0,1]$ such that $\sum_{y_i=1}^m a_t(y_0,\ldots,y_t)=1$ and
\begin{align*}
\nu_t^{y_0\ldots y_{t-1}}(d \tilde y_t) = \sum_{y_t=1}^m a_t(y_0,\ldots,y_t)\delta_{y_t}(d \tilde y_t).
\end{align*}
Next, we may find measurable functions $\tilde a_t : \Y_0 \times \{1,\ldots,m\}^t \to [0,1]$ for $t=1,\ldots,N$, each with finite range, satisfying $\sum_{y_t=1}^m \tilde a_t(y_0,\ldots,y_t)=1$, and also for every $y \in \Y_0 \times \{1,\ldots,m\}^t$ one of the following conditions holds:
\begin{align}
\begin{split}
\mbox{Case A:} \quad & a_t(y) > \eta - 1 \ \text{ and } \ \tilde a_t(y)/\eta \leq a_t(y)\leq \tilde a_t(y) \eta, \\
\mbox{Case B:} \quad & 0 < a_t(y) \leq \eta-1 \ \text{ and } \  \tilde a_t(y) = \eta-1, \\
\mbox{Case C:} \quad & a_t(y)=0 \ \text{ and } \ \tilde a_t(y) = 0.
\end{split} \label{eq:Cases-AB}
\end{align}
We will use this to build a new joint distribution $\tilde\gamma \in \P(\Y \times \X)$, where we recall the notation $\X=\X_0\times\cdots\times\X_N$ and $\Y=\Y_0\times\cdots\times\Y_N$. But this requires a bit of notation.

For each $t \in \{0,\ldots,N\}$, let $\gamma_{\le t} \in \P(\Y_0\times\X_0 \times \cdots \times \Y_t\times\X_t)$ denote the projection onto the first $t+1$ coordinates. Define $\nu_{\le t}\in\P(\Y_0\times\cdots\times\Y_t)$ and $\mu_{\le t}\in\P(\X_0\times\cdots\times\X_t)$ analogously.
Adopting the convention $0/0:=0$, we next define $D_t : \Y \to [0,\infty)$ for $t \in \{1,\ldots,N\}$ by
\begin{align*}
D_t(y) := \prod_{i=1}^t\frac{\tilde a_i(y_0,\ldots,y_i)}{a_i(y_0,\ldots,y_i)}, \quad y=(y_0,\ldots,y_N) \in \Y.
\end{align*}
Next, define a stopping time $\tau : \Y \to \{1,\ldots,N\} \cup \{\infty\}$ by setting
\begin{align*}
\tau(y_0,\ldots,y_N) &:= \min\{t \in \{1,\ldots,N\} : a_t(y_0,\ldots,y_t) \le \eta -1\},
\end{align*}
with $\min\emptyset=\infty$.
We note for later use that $\nu(\tau < \infty) \le Nm(\eta-1)$, because
\begin{align*}
\nu(\tau < \infty) &= \nu(\exists t :  a_t \le \eta - 1) \le \sum_{t=1}^N\nu(a_t \le \eta - 1)
\end{align*}
and also
\begin{align*}
\nu_t^{y_0\ldots y_{t-1}}\big(a_t(y_0,\ldots,y_{t-1},\cdot) \le \eta - 1\big) = \sum_{y_t=1}^m a_t(y_0,\ldots,y_t)1_{\{a_t(y_0,\ldots,y_t) \le \eta-1\}} \le m(\eta-1).
\end{align*}

Next define $\tilde \nu \in \P(\Y)$ by $d\tilde\nu/d\nu = D_N$, and note that $\tilde\nu$ admits the disintegration
\begin{align*}
\tilde\nu(dy_0,\ldots,dy_N) &= \lambda(dy_0)\tilde\nu_1^{y_0}(dy_1)\cdots\tilde\nu_N^{y_0\ldots y_{N-1}}(dy_N), \qquad \tilde\nu_t^{y_0\ldots y_{t-1}} = \sum_{y_t=1}^m a_t(y_0,\ldots,y_t)\delta_{y_t}.
\end{align*} 
We finally define $\tilde\gamma \in \P(\Y \times \X)$ by
\begin{align}
&\tilde\gamma(dy,dx) = 1_{\{\tau=\infty\}}D_N(y)\gamma(dy,dx) \label{def:gammatilde} \\
	& \quad + \sum_{t=1}^N 1_{\{\tau = t\}} D_{t-1}(y)\gamma_{\le t-1}(dy_0,dx_0,\ldots,dy_{t-1},dx_{t-1})\prod_{i=t}^N\tilde\nu_i^{y_0\ldots y_{i-1}}(dy_i) \, \mu_i^{x_0\ldots x_{i-1}}(dx_i). \nonumber
\end{align} 
Note that $\tilde\nu$ is precisely the $\Y$-marginal of $\tilde\gamma$, and we let  $\tilde\mu \in \P(\X)$ denote the $\X$-marginal of $\tilde\gamma$. Note that
\begin{align*}
\frac{d \nu}{d\tilde \nu}(y) = \frac{1}{D_N(y)} \le \eta^N,
\end{align*}
since $a_t \le \eta \tilde a_t$ pointwise. Moreover, we claim that
\begin{align*}
d\tilde\mu/d\mu \le \eta^N + mN\eta^N(\eta-1).
\end{align*}
To see this, choose a measurable function $f : \X \to [0,1]$ and $t \in \{1,\ldots,N\}$. Note that on $\{\tau=t\}$ it holds that $a_t\le \eta-1$ and thus $\tilde a_t=\eta-1$, and in particular $\tilde\nu_t^{y_0,\ldots,y_{t-1}} \le (\eta-1)\sum_{\ell=1}^m \delta_\ell$. In addition, on $\{\tau=t\}$ we have $a_i > \eta -1$ and for $i \le t-1$ and thus $D_{t-1} \le \eta^{t-1} \le \eta^N$. Similarly, on $\{\tau=\infty\}$ we have $D_N \le \eta^N$. Combining these observations with the formula \eqref{def:gammatilde}, a quick but notationally awkward calculation yields
\begin{align*}
\int_\X f\,d\tilde\mu &= \int_{\Y \times \X} f(x) \,\tilde\gamma(dy,dx) \le (\eta^N + mN\eta^N(\eta-1))\int_\X f\,d \mu.
\end{align*}

Now, writing
\begin{align*}
\tilde\nu^{y_0}_+(dy_1,\ldots,dy_N) := \tilde\nu_1^{y_0}(dy_1)\cdots\tilde\nu_N^{y_0\ldots y_{N-1}}(dy_N), \quad \text{for } y_0 \in \Y_0=[0,1],
\end{align*}
we may write $\tilde\nu=\lambda(dy_0)\tilde\nu_1^{y_0}(dy_1,\ldots,dy_N)$. Recalling that the functions $\tilde a_t$ have finite range, the range of the map $y_0 \mapsto \tilde\nu_1^{y_0}$ is also finite. Hence, there is a finite partition $\Pi$ of $[0,1]$ into measurable sets such that $y_0 \mapsto \tilde\nu_1^{y_0}$ is constant when restricted to any set in $\Pi$. Then, for each $S \in \Pi$ with $\lambda(S) > 0$, the measure
\[
\tilde\nu_{|S \times \Y_1\times\cdots\times\Y_N} = \lambda_{|S}(dy_0)\nu_1^{y_0}(dy_1,\ldots,dy_N)
\]
on $S \times [0,1]$ is a product measure. Thus, $\tilde\nu$ is of the form specified in Proposition \ref{pr:density-discrete-specialcase2}. We may thus find sequences of measurable functions $U^k : \Y_0 \to [0,1]$ and $f^k_t : \Y_0 \times \cdots \times \Y_t \to \X_t$ for $t\in\{0,\ldots,N\}$ such that:
\begin{itemize}
\item $U^k$ is independent of $f^k=(f^k_0,\ldots,f^k_N)$ under $\tilde\nu$ for each $k \in \N$.
\item $\tilde\nu \circ (U^k)^{-1} = \lambda$ for each $k \in \N$.
\item $\tilde\nu \circ (f^k)^{-1}=\tilde\mu$ for each $k \in \N$.
\item $\tilde\gamma^k := \tilde\nu \circ (Id,f^k)^{-1}$ converges weakly to $\tilde\gamma$ as $k\to\infty$.
\end{itemize}
Here we write $f^k$ for the function $\Y\to\X$ given by $f^k(y) = (f^k_0(y_0),f^k_1(y_0,y_1),\ldots,f^k_N(y_0,\ldots,y_N))$ for $y=(y_0,\ldots,y_N) \in \Y$, and $(Id,f^k)$ is the function $\Y \to \Y \times \X$ given by $y \mapsto (y,f^k(y))$.

We finally use the additional independent uniforms $U^k$ to modify $\tilde\gamma^k$ to have marginals $(\nu,\mu)$ instead of $(\tilde\nu,\tilde\mu)$. We will abuse notation somewhat by viewing $f^k_t$ and $U^k$ as functions on $\Y$. In the following we will encounter the peculiar constant
\[
\tilde\eta := \eta^{2N}(1 + mN(\eta-1)).
\]
Note that $\tilde\eta > 1$ and $\tilde\eta \to 1$ as $\eta \to 1$.

Let $E_k = \{U^k \le 1/\tilde\eta\}$, which we somewhat abusively view as a subset of either $\Y_0=[0,1]$, $\Y$, or $\Y \times \X$.
Using $\widetilde\nu \circ (f^k)^
{-1} = \widetilde\mu$ and the independence of $U^k$ and $f^k$, we deduce
\begin{align*}
\widetilde\nu_{|E_k} \circ (f^k)^{-1} \le \frac{1}{\tilde\eta}\widetilde\mu.
\end{align*}
Recalling from above that $d\nu/d\tilde\nu \le \eta^N$ and $d\tilde\mu/d\mu \le \eta^N + mN\eta^N(\eta-1)$, this implies
\begin{align*}
\bar\mu &:= \nu_{|E_k} \circ (f^k)^{-1} \le \eta^N \, \widetilde\nu_{|E_k} \circ (f^k)^{-1} \le \frac{\eta^N}{\tilde\eta}\,\widetilde\mu \le \mu.
\end{align*}
Note that $\mu-\bar\mu$ has total mass $1-\frac{1}{\tilde\eta}$, and so does the measure $\lambda_{|E_k^c}$. By Borel isomorphism, we may thus find a measurable function $T^k=(T^k_0,\ldots,T^k_N) : \Y_0 \to \X$ such that
\begin{align*}
\lambda_{|E_k^c} \circ (T^k)^{-1} = \mu - \bar\mu.
\end{align*} 
We then define $g^k_t : \Y_0 \times \cdots \times \Y_t \to \X_t$ by
\begin{align*}
g^k_t(y_0,\ldots,y_t) := \begin{cases}
f^k_t(y_0,\ldots,y_t) &\text{if } U^k(y_0) \le \frac{1}{\tilde\eta} \\
T^k_t(y_0) &\text{if } U^k(y_0) > \frac{1}{\tilde\eta}.
\end{cases}
\end{align*}
Setting $g^k=(g^k_0,\ldots,g^k_N)$, we then find that $\nu \circ (g^k)^{-1} = \mu$. Indeed,
\begin{align*}
\nu \circ (g^k)^{-1} &= \nu_{|E_k} \circ (f^k)^{-1} + \lambda_{|E_k^c} \circ (T^k)^{-1}  = \bar\mu + \mu - \bar\mu = \mu.
\end{align*}
Finally, let $\gamma^k := \nu \circ (Id,g^k)^{-1}$. Recalling that $\widetilde\gamma$ is close to $\gamma$ as $\eta \to 1$, we see that $\gamma^k_{|E_k}$ is close to $\widetilde\gamma^k_{|E_k}$, which is in turn close to $\gamma_{|E_k}$. The complementary event $E_k^c$ is negligible as $k\to\infty$, and we conclude that $\gamma^k \to \gamma$ as $k\to\infty$ and $\eta \to 1$. Extract a diagonal subsequence to complete the proof.
\end{proof}

\subsection{Step 4: General case}

Finally we establish a denseness result which reduces the desired result to the case treated in the previous section.  In this part we write $\pi_2\circ \pi_1$ / $ \pi\circ T$ for concatenations of transport plans with transport plans and of transport plans with transport maps respectively. Precisely, if $\pi_1 \in \Pi(\nu,\mu)$ and $\pi_2 \in \Pi(\mu,\gamma)$, then $\pi_2 \circ \pi_1 \in \Pi(\nu,\gamma)$ is defined by $\pi_2 \circ \pi_1(A) = \int \int 1_A(x,z)\pi_1(dx,dy)\pi_2(y;dz)$ where we disintegrate $\pi_2(dy,dz)=\mu(dy)\pi_2(y;dz)$. Similarly, if $\pi \in \Pi(\nu,\mu) \subset \P(E \times F)$ and $T : F \to G$ is measurable, then $\pi \circ T := \pi \circ (Id,T)^{-1}$.

For $m\in\N$, consider the mapping 
\begin{align}\label{eq:rounding}
S_m:[0,1]^{N+1} \to [0,1] \times \{1/m, 2/m, \ldots, m/m\}^N \subset [0,1]^{N+1}, \\
S_m( y_0, y_1,  \ldots,  y_n):=( y_0, \lceil y_1 m\rceil/m, \ldots , \lceil  y_n m\rceil/m)
\end{align}
Let $\nu \in \P([0,1]^{N+1})$, and set $\nu^m:= \nu \circ S_m^{-1}$. Abbreviate by $S_m^\nu:=\nu \circ (Id,S_m)^{-1}$ the natural coupling of $\nu$ and $\nu^m$, which belongs to $\Pi^c(\nu,\nu^m)$.
For $\mu \in \P([0,1]^{N+1})$ denote by $\Pi^{c,m}(\nu,\mu)$ the set of couplings in $\Pi^c(\nu,\mu)$ which are of the form $\pi\circ S_m^\nu$ for some $\pi\in \Pi^c( \nu^m,\mu)$, i.e.,
\[
\Pi^{c,m}(\nu,\mu) := \{ \pi\circ S_m^\nu : \pi\in \Pi^c(\nu^m,\mu)\}.
\]

Clearly $\nu^m$ converges weakly to $\nu$ as $m\to \infty$. This is of course equivalent to convergence $W(\nu^m,\nu)\to 0$, where $W$ is the Wasserstein distance on $\P([0,1])$. This is in turn equivalent to saying that there is a sequence of couplings $\pi^m \in \Pi(\nu^m,\nu)$ which converges weakly to the identical coupling. But we are after something stronger, namely the existence of a sequence of \emph{causal} couplings converging to the identical coupling. We achieve this by showing that $\nu^m\to\nu$ in a stronger metric known as the \emph{bi-causal Wasserstein distance} (or `nested distance'). Recall that the set of bi-causal transport plans consists of
\begin{align}
\Pi^{bc}(\nu,\mu):= \Pi^c(\nu,\mu) \cap \{\pi \circ r^{-1} : \pi \in \Pi^c(\mu,\nu)\}, \label{def:bicausalKanto}
\end{align}
where $r(y,x):=(x,y)$. The bi-causal $1$-Wasserstein distance is defined by
\begin{align}
W_{bc}(\nu,\mu):= \inf_{\pi \in \Pi^{bc}(\nu,\mu)} \int_{[0,1]^{N+1} \times [0,1]^{N+1}} \sum_{i=0}^N |y_i-\tilde y_i|\, \pi(dy,d\tilde y). \label{def:Wbicausal}
\end{align}

\begin{lemma}\label{le:specfic-rounding-lemma}
Using the above notations, we have $\lim_{m\to \infty} W_{bc}(\nu, \nu^m)=0$.
\end{lemma}
The proof of Lemma \ref{le:specfic-rounding-lemma} is of a rather different spirit than the rest of this part and we postpone it to Section \ref{se:ApproxSection}.
Here we use Lemma \ref{le:specfic-rounding-lemma} to show the following:

\begin{lemma}\label{le:AuxDenseness} Using the above notations, 
the set $\bigcup_m\Pi^{c,m}(\nu,\mu)$ is dense in  $\Pi_c(\nu,\mu)$.
\end{lemma}
\begin{proof}[Proof of Lemma \ref{le:AuxDenseness}]
We first note the if $\{\tau_m : m \in \N\} \subset \Pi(\nu,\nu)$ is a sequence of couplings which converges weakly to the identical coupling and $\pi \in \Pi(\nu,\mu)$, then also $\pi \circ \tau_m \Rightarrow  \pi$. 

Write $\sigma_m\in \Pi^{bc}(\nu^m,\nu)$ for a coupling which realizes $W_{bc}(\nu^m,\nu)$. Indeed, the infimum is attained because $\Pi^{bc}(\nu,\mu)$ is closed subset of the compact set $\Pi(\nu,\mu)$. By Lemma \ref{le:specfic-rounding-lemma}, $\int|y-x|\, d\sigma_m(y,x)\to 0$ as $m\to\infty$. 
Since $S_m^\nu$ as well as $\sigma_m$ converge to the identical coupling, so does $\tau_m:= \sigma_m\circ S_m^\nu$. Given $\pi\in \Pi^c(\nu,\mu)$, we thus find
\begin{align*}
\pi= \lim_m \pi \circ \tau_m = \lim_m (\pi\circ \sigma_m ) \circ S_m^\nu.
\end{align*}
\end{proof}

We can now finally establish our main result:

\begin{proof}[Proof of Theorem \ref{th:intro:density-discrete}.] Using a Borel isomorphism, we may assume that $\Y_0=\Y_1=\cdots=\Y_N=[0,1]$. By Proposition \ref{pr:density-discrete-specialcase3}, the set of causal Monge couplings is dense in $\Pi^{c,m}(\nu,\mu)$ for each $m$. That is,
\[
\Pi^{c,m}_0(\nu,\mu) := \{\pi \circ S_m^\nu : \pi \in \Pi^c_0(\nu^m,\mu)\}
\]
is dense in $\Pi^{c,m}(\nu,\mu)$.
By Lemma \ref{le:AuxDenseness},  $\bigcup_m \Pi^{c,m}(\nu,\mu)$ is dense in $\Pi^c(\nu,\mu)$.
\end{proof}

\subsection{On the bi-causal approximation}\label{se:ApproxSection}
In this section we show that probabilities concentrated on finite sets are dense in the set of all distributions with respect to the bi-causal Wasserstein distance.
We will work here in the general setting described at the beginning of Section \ref{se:discretetime-twomarg}, assuming that $\Y_0,\ldots,\Y_N$ are now compact metric spaces. (Compactness can be relaxed, but this will suffice for our purposes.)
For $\Y:=\Y_0\times\cdots\times\Y_N$ and $\nu,\mu \in \P(\Y)$, we define $\Pi^c(\nu,\mu)$ as in the beginning of the section (with $\X_i=\Y_i$ for all $i$).
Define $\Pi^{bc}(\nu,\mu)$ as in \eqref{def:bicausalKanto}.
Let $d_i$ denote a bounded metric on $\Y_i$, for each $i=0,\ldots,N$. Define the bicausal distance (also called nested distance) on $\P(\Y)$ by
\begin{align}\label{AWrec}
W_{bc}(\nu,\mu) := \inf_{\gamma \in \Pi^{bc}(\nu,\mu)} \int_{\Y\times\Y} \sum_{i=0}^N d_i(y_i,\tilde y_i)\,\gamma(dy,d\tilde y).
\end{align}
The bicausal distance was introduced by Pflug and Pflug-Pichler \cite{Pf09, PfPi12}.

Recall in the following that an $\epsilon$-net of a metric space $\X$ is a subset $E \subset \X$ with the property that the every point in $\X$ has distance no more than $\epsilon$ from the $E$.

\begin{theorem}\label{thm:TreeApprox}
For each  $m\in \N$ and $i \in\{0,\ldots,N\}$, let $E^m_i$ be a $(1/m)$-net of $\Y_i$ with finite cardinality, and define $\varphi^m_i : \Y_i \to E^m_i$ by setting $\varphi^m_i(y_i)$ to be the nearest point of $E^m_i$ to $y_i$, with ties broken arbitrarily (but measurably).
Define $R_m: \Y \to E^m_0 \times \ldots \times E^m_N$ by 
\begin{align}
R_m( y_0, \ldots, y_N):= (\varphi_0^m(y_1), \ldots, \varphi_N^m(y_N)) .
\end{align}
Then, for any $\nu \in \P(\Y)$, we have
\begin{align}
\lim_{m\to\infty} W_{bc}( \nu, \nu \circ R_m^{-1}) = 0.
\end{align}
\end{theorem}
\begin{proof}
Throughout the proof we write $W$ for the Wasserstein distance on $\P(\Y_i)$ as usual by
\begin{align*}
W(\nu_i,\mu_i) := \inf_{\gamma \in \Pi(\nu_i,\mu_i)} \int_{\Y_i \times \Y_i} d_i(y_i,\tilde y_i)\,\gamma(dy_i,d\tilde y_i), \qquad i \in \{1,\ldots,N\}, \ \ \nu_i,\mu_i \in \P(\Y_i).
\end{align*}
Importantly, by \cite[Theorem 1]{BaBaBeEd19b} it is sufficient to establish convergence in \emph{Hellwig's information topology}. We do not need the full definition, only to note that it suffices to consider the two-period case $N=1$, and it is hard to overstate the value of this simplification. We can also assume that the metrics $d_0$ and $d_1$ on $\Y_0$ and $\Y_1$ are both bounded by $1$.
Throughout the proof, we disintegrate
\begin{align*}
\nu(dy_0,dy_1) = \nu_0(dy_0)\nu_1^{y_0}(dy_1)
\end{align*}
We also set $C_i(e):= \varphi_i^{-1}(\{e\})$, noting that $C_i(e)$ has diameter no more than $1/m$ since $E_i$ is a $(1/m)$-net. 

Fix $\epsilon>0$. Applying Lusin's theorem to the mapping $y_0\to \nu_1^{y_0}$, there is a compact set $K\subseteq \Y_0$ such that $\nu_0(K)>1-\epsilon^2$ and $y_0\mapsto \nu_1^{y_0}$ is (uniformly) continuous on $K$ with respect to Wasserstein-1 distance.
 
Pick $m$ large enough so that $y_0\to \nu_1^{y_0}$ varies by less than $\epsilon$ on $C_0(e_0) \cap K$ for each $e_0 \in E^m_0$ (precisely, $\max_i\sup_{y_0,y_0' \in C_0(e_0)\cap K}W(\nu_1^{y_0},\nu_1^{y_0'}) \le \epsilon$), 
and such that $1/m \leq \epsilon$. From now on $m$ will remain fixed and we will suppress it in certain places to avoid overly heavy notation. In particular we will just write $E_i$ and $\varphi_i$ instead of $E_i^m$ and $\varphi_i^m$. 

Let $\mu := \nu \circ R_m^{-1}$, and disintegrate
\[
\mu(de_0,de_1) = \mu_0(de_0)\mu_1^{e_0}(de_1).
\]
As $\mu$ and its disintegrations are supported on finite sets, we will write, e.g., $\mu(e_0)$ in place of $\mu_0(\{e_0\})$  for $e_0 \in E_0$.
Note that
\[
\mu_0(e_0)= \nu_0(C_0(e_0)), \qquad \mu     _1^{e_0}(e_1)=  \nu\big(C_0(e_0)\times C_1(e_1)\big) / \mu_0(e_0), \ \text{ when } \ \mu_0(e_0) > 0.
\]

Let us define $G_0$ to be the set of $e_0 \in E_0$ satisfying $\nu_0(C_0(e_0)\cap K) \ge \nu_0(C_0(e_0))(1-\epsilon)$, or equivalently $\nu_0(C_0(e_0) \cap K^c) \le \epsilon\,\nu_0(C_0(e_0))$.
Note that $\mu_0(G_0) = \nu_0(\cup_{e_0 \in G_0}C_0(e_0)) \ge 1-\epsilon$, because
\begin{align*}
\nu_0(\cup_{e_0 \in G_0^c}C_0(e_0)) &= \sum_{e_0\in G_0^c} \nu_0(C_0(e_0)) < \frac{1}{\epsilon}\sum_{e_0\in G_0^c}\nu_0(C_0(e_0)\cap K^c) \le \frac{1}{\epsilon}\nu_0(K^c) \le \epsilon.
\end{align*}
For $e_0 \in G_0$ and $\tilde y_0 \in C_0(e_0) \cap K$, it holds that $W(\nu_1^{\tilde y_0},\mu_1^{e_0}) \le 3\epsilon$.
Indeed, this follows from the inequalities
\begin{align*}
W\left( \mu_1^{e_0}, \frac1{\nu_0(C_0(e_0))} \int_{C_0(e_0)} \nu_1^{y_0}\, \nu_0(dy_0) \right) &\le \frac1{\nu_0(C_0(e_0))} \int_{C_0(e_0)} \int_{\Y_1} |y_1 - \varphi_1(y_1)| \nu_1^{y_0}(dy_1)\, \nu_0(dy_0) \\
&\le 1/m \le \epsilon 
\end{align*} 
and
\begin{align*}
W\left(\nu_1^{\tilde y_0}, \frac1{\nu_0(C_0(e_0))} \int_{C_0(e_0)} \nu_1^{y_0}\, \nu_0(dy_0) \right) &\le 
\frac1{\nu_0(C_0(e_0))} \int_{C_0(e_0)}  W\left(\nu_1^{\tilde y_0} , \nu_1^{y_0}\right)\, \nu_0(dy_0) \\
	&\le \frac{\nu_0(C_0(e_0) \cap K^c)}{\nu_0(C_0(e_0))} + \epsilon \\
	&\le 2\epsilon,
\end{align*}
where we used the facts that $\Y_1$ has diameter at most 1 and that $y_0\mapsto\nu_1^{y_0}$ varies by at most $\epsilon$ on $C_0(e_0) \cap K$.

We use next the dynamic programming principle for bicausal optimal transport (this is stated explicitly in e.g.\ \cite[Section 3.1]{BaBeEdPi17} but in fact goes back to much earlier work of R\"uschendorf \cite[Theorem 3]{Ru85}), which in our two-period setting is the identity
\begin{align*}
W_{bc}(\nu,\mu) = \inf_{\gamma_0 \in \Pi(\nu_0,\mu_0)}\int_{\Y_0\times E_0} \big(d_0(y_0,e_0) + W(\nu_1^{y_0},\mu_1^{e_0}) \big) \gamma_0(dy_0,de_0).
\end{align*}
The distance $W(\nu_1^{\tilde y_0},\mu_1^{e_0})$ is attained by a coupling $\gamma_1^{\tilde y_0 e_0} \in \Pi(\nu_1^{\tilde y_0},\mu_2^{e_0})$, which by standard arguments can be chosen to be measurable in $(\tilde y_0,e_0)$. 
Choosing $\gamma_0(dy_0,de_0) := \nu_0(dy_0)\delta_{\varphi_0(y_0)}(de_0)$ in the above identity, and recalling that $1/m \le\epsilon$ and $\sum_{e_0 \in G_0^c}\nu_0(C_0(e_0)) \le \epsilon$, we find
\begin{align*}
W_{bc}(\nu,\nu \circ R_m^{-1}) &= W_{bc}(\nu,\mu) \le \int_{\Y_0} \big(d(y_0,\varphi_0(y_0)) + W(\nu_1^{y_0},\mu_1^{\varphi_0(y_0)})\big)\,\nu_0(dy_0) \\
	&\le 1/m + \int_{\Y_0} W(\nu_1^{y_0},\mu_1^{\varphi_0(y_0)})\,\nu_0(dy_0)  \\
	&\le \epsilon + \sum_{e_0 \in E_0}\int_{C_0(e_0) }  W(\nu_1^{y_0},\mu_1^{e_0})\,\nu_0(dy_0) \\
	&\le 2\epsilon + \sum_{e_o \in G_0}\int_{C_0(e_0) } W(\nu_1^{y_0},\mu_1^{e_0})\,\nu_0(dy_0)  \\
	&\le 2\epsilon + \epsilon^2 + \sum_{e_o \in G_0}\int_{C_0(e_0) \cap K} W(\nu_1^{y_0},\mu_1^{e_0})\,\nu_0(dy_0) \\
	&\le 2\epsilon + \epsilon^2 + 3\epsilon\sum_{e_o \in G_0}\nu_0(C_0(e_0) \cap K) \\
	&\le 5\epsilon + \epsilon^2.
\end{align*} 
This completes the proof.  
\end{proof}

Finally are ready to prove  Lemma \ref{le:specfic-rounding-lemma}:

\begin{proof}[Proof of Lemma \ref{le:specfic-rounding-lemma}]
Set $\Y_i=[0,1]$ for $i=0, \ldots, N$.
Recall from the previous section that for $m\geq 1$ the mapping $S_m$ is given by  
\begin{align*}
S_m:[0,1]^{N+1} & \to [0,1] \times \{1/m, 2/m, \ldots, m/m\}^N, \\
S_m(y_0,y_1,\ldots,y_N) & :=(y_0,\lceil y_1 m\rceil/m, \ldots, \lceil y_N m\rceil/m).
\end{align*} 
We define $R_m$ through
\begin{align*}
R_m:[0,1]^{N} & \to \{1/m, 2/m, \ldots, m/m\}^N, \\
R_m( y_1,  \ldots,  y_N) & :=(\lceil  y_1 m\rceil/m, \ldots , \lceil  y_n m\rceil/m).
\end{align*} 
Note that $S_m(y_0,y_1,\ldots,y_N) = (y_0,R_m(y_1,\ldots,y_N))$.
Then $\nu$ admits the disintegration
\begin{align*}
\nu(dy_0,\ldots,dy_N) = \nu_0(dy_0)\nu^{ y_0}(dy_1,\ldots,dy_N),
\end{align*}
and $\nu \circ S_m^{-1}$ admits the disintegration
\begin{align*}
\nu \circ S_m^{-1}(dy_0,\ldots,dy_N) = \nu_0(dy_0) \nu^{y_0} \circ R_m^{-1}(dy_1,\ldots,dy_N).
\end{align*}
We use again the dynamic programming principle for causal optimal transport \cite[Theorem 2.7]{veraguas2016causal} to write
\begin{align}
W_{bc}(\nu, \nu\circ S_m^{-1}) &= \inf_{\gamma\in \Pi(\nu_0, \nu_0 \circ S_m^{-1})}\int_{[0,1] \times [0,1]} |y_0 - x_0| + W_{bc}(\nu^{y_0},  \nu^{x_0} \circ R_m^{-1}) \, \gamma(dy_0,d \tilde y_0)  \\
&  \leq \int_{[0,1]} W_{bc}(\nu^{y_0}, \nu     ^{y_0} \circ R_m^{-1}) \, \nu_0(dy_0).  
\end{align}
By Theorem \ref{thm:TreeApprox} and dominated convergence, this quantity tends to $0$ as $m\to \infty$.
\end{proof}

\section{Continuous time} \label{se:dynamic-cont}

\subsection{Continuous time, continuous paths} \label{se:continuoustime}
We now extend the results of the previous section to continuous time. There several natural choices of path space for our processes, and we will work mainly with continuous and c\`adl\`ag paths. Let $C(\R_+;\X)$ denote the space of continuous functions from $\R_+$ to $\X$, endowed with the topology of uniform convergence on compacts.
Let $D(\R_+;\X)$ denote the Skorohod space of c\`adl\`ag functions\footnote{As usual, a function $f$ defined on an interval in $\R_+$ is called c\`adl\`ag if it is right-continuous and the limit $\lim_{s \uparrow t}f(s)$ exists for every $t > 0$.} from $\R_+$ to (a metric space) $\X$, endowed with the usual Skorohod $J_1$ topology. Define $C([0,T];\X)$ and $D([0,T];\X)$ analogously on finite time intervals.

For a continuous-time stochastic process $X=(X_t)_{t \ge 0}$, let $\FF^X=(\F^X_t)_{t \ge 0}$ denote the filtration generated by $X$, i.e., $\F^X_t = \sigma(X_s : s \le t)$. Recall that a process $Y=(Y_t)_{t \ge 0}$ is \emph{$\FF^X$-adapted} if $Y_t$ is $\F^X_t$-measurable for each $t \ge 0$. For any filtration $\FF=(\F_t)_{t \ge 0}$, define as usual
\[
\F_\infty = \sigma\bigg(\bigcup_{t \ge 0}\F_t\bigg).
\]
Filtrations are not augmented in any way unless explicitly stated.

\begin{assumption*}[\textbf{A}] \label{ass:A}
We are given a probability space $(\Omega,\F,\PP)$ supporting the following:
\begin{enumerate}[(1)]
\item A random variable $Z$ taking values in a Polish space $\Z$.
\item A filtration $\FF=(\F_t)_{t \ge 0}$ satisfying the following:
\begin{enumerate}[(a)]
\item $\F_t \subset \sigma(Z)$, for every $t \ge 0$.
\item $(\Omega,\F_t)$ is a standard Borel space  for each $t \ge 0$.
\item $(\Omega,\F_t,\PP)$ is nonatomic  for each $t > 0$.
\end{enumerate}
\item A measurable stochastic process $X=(X_t)_{t \ge 0}$ taking values in a Polish space $\X$.
\end{enumerate}
\end{assumption*}

An additional bit of structure will allow us to add another useful equivalence in the following results and streamline the proof:

\begin{assumption*}[\textbf{B}] \label{ass:B}
We are given for each $t \ge 0$ a random variable $U_t$ with values in a Polish space $\U_t$ such that $\F_t = \sigma(U_t)$.
\end{assumption*}

It is worth noting that, with Assumption \hyperref[ass:A]{(A)} in force, Assumption \hyperref[ass:B]{(B)} always holds with $\U_t=[0,1]$ for every $t$, thanks to Borel isomorphism and Assumption (A.2b). In other words, Assumption \hyperref[ass:B]{(B)} imposes no additional structural constraints. Note that we make no requirements of measurability in $t$ of the process $(U_t)$.

\begin{example} \label{ex:Ycadlag}
Typically, we are given a c\`adl\`ag process $Y=(Y_t)_{t \ge 0}$ with values in some Polish space $\Y$ such that $\F_t = \F^Y_t = \sigma(Y_s : s \le t)$ is the natural filtration, and also $\Z=D(\R_+;\Y)$ and $Z=Y$. In Assumption \hyperref[ass:B]{(B)}, we may then set $\U_t = D([0,t];\Y)$ and $U_t = Y|_{[0,t]}$. With these identifications, we deduce Theorem \ref{th:intro:density-continuous} from Theorem \ref{th:density-continuous}.
\end{example}

One last technical assumption is needed to avoid difficulties at time zero:

\begin{assumption*}[\textbf{C}] \label{ass:C}
Either $X_0$ is a.s.\ $\F_0$-measurable, or $(\Omega,\F_0,\PP)$ is nonatomic.
\end{assumption*}

\begin{theorem}[Continuous paths] \label{th:density-continuous}
Suppose Assumptions \hyperref[ass:A]{(A)}, \hyperref[ass:B]{(B)}, and \hyperref[ass:C]{(C)} hold, and that the process $X$ of Assumption \hyperref[ass:A]{(A.3)} has continuous paths. Assume also that the space $\X$ is homeomorphic to a convex subset of a locally convex space. Then the following are equivalent:
\begin{enumerate}[(i)]
\item For each $t \ge 0$,  $\F^X_t$ is conditionally independent of $Z$ given $\F_t$.
\item There exists a dense set $\T \subset \R_+$ such that, for each $t \in \T$,  $\F^X_t$ is conditionally independent of $Z$ given $\F_t$.
\item There exists a sequence $X^n$ of continuous $\FF$-adapted processes such that $(X^n,Z) \Rightarrow (X,Z)$ in $C(\R_+;\X) \times \Z$.
\item Statement (iii) holds, and, for each $n$, $X^n$ is linear between some time points $0=t_0 < t_1 < \cdots < t_n$, and constant after time $t_n$, with $X^n_{t_k} = f_k(U_{t_{k-1}})$ for some continuous function $f_k : \U_{t_{k-1}} \rightarrow \X$, for each $k=1,\ldots,n$.
\end{enumerate}
\end{theorem}

The role of the random variable $Z$ in Theorem \ref{th:density-continuous}  is to incorporate additional information that the process $X$ must respect. This can be useful, for example, for stochastic control problems with partial information, illustrated in Section \ref{se:optimalcontrol}. 
Note that the weak convergence $(Z,X^n) \Rightarrow (Z,X)$ in $\Z \times C(\R_+;\X)$ implies the weak convergence $(Y,X^n) \Rightarrow (Y,X)$ in $D(\R_+;\Y) \times C(\R_+;\X)$ for any c\`adl\`ag $Z$-measurable process $Y=(Y_t)_{t \ge 0}$ with values in some Polish space $\Y$, which follows quickly from Lemma \ref{le:stable-convergence}.

\begin{proof}[Proof of Theorem \ref{th:density-continuous}]
Clearly (i) implies (ii) and (iv) implies (iii).
We next prove that (ii) implies (iv). First, using the fact that $\X$ is a convex subset of a vector space, it is straightforward to construct a sequence of continuous $\FF^X-$adapted processes, $X^n$, such that $X^n \rightarrow X$ a.s. in $C(\R_+;\X)$ and also each $X^n$ is piecewise linear with a deterministic and finite set of points of non-differentiability. Thus, it suffices to prove the claim under the assumption that $X$ is of the form 
\begin{align*}
 X_t &= \frac{t-t_k}{t_{k+1}-t_k}H_k+ \frac{t_{k+1}-t}{t_{k+1}-t_k}H_{(k-1)^+}, \quad \text{ for } t \in [t_k,t_{k+1}], \ k=0,\ldots,m-1, \\
 X_t &= H_{m-1}, \quad \text{ for } t \ge t_m,
\end{align*}
where $m \in \N$, the times $0=t_0 < t_1 < \ldots < t_m$ are deterministic, and $H_0,\ldots,H_{m-1}$ are some $\X$-valued random variables.
Assumption (ii) implies that $H_k$ is conditionally independent of $Z$ given $\F_{t_k}=\sigma(U_{t_k})$ for each $k=0,\ldots,m-1$.

We now apply Theorem \ref{th:density-discrete} with $Y_k=U_{t_k}$ for each $k=1,\ldots,m-1$ and $Y_m=Z$. Note that the process $Y=(Y_1,\ldots,Y_m)$ takes values in the disjoint union $\Y=\sqcup_{k=1}^{m-1}\U_{t_k} \sqcup \Z$. Moreover, $\F^Y_k = \F_{t_k}$ for $k=1,\ldots,m-1$ and $\F^Y_m=\sigma(Z)$. Apply Theorem \ref{th:density-discrete} to find a sequence $H^n=(H^n_k)_{k=1}^m$ of $(\F^Y_k)_{k=1}^m$-adapted processes such that
\[
\left(Y,H^n_1,\ldots,H^n_m\right) \Rightarrow \left(Y,H_1,\ldots,H_m\right),
\]
in $\Y^m \times \X^m$. Discarding the first $m-1$ components of $Y=(Y_1,\ldots,Y_m)$ and using $Y_m=Z$, we deduce 
\[
\left(Z,H^n_1,\ldots,H^n_{m-1}\right) \Rightarrow \left(Z,H_1,\ldots,H_{m-1}\right),
\]
in $\Z \times \X^{m-1}$. 
In Assumption \hyperref[ass:C]{(C)}, if $X_0=H_0$ is $U_0$-measurable, then setting $H^n_0 := H_0$ we may use Lemma \ref{le:stable-convergence} to get
\begin{align}
\left(Z,H^n_0,H^n_1,\ldots,H^n_{m-1}\right) \Rightarrow \left(Z,H_0,H_1,\ldots,H_{m-1}\right), \label{pf:density-continuous1}
\end{align}
in $\Z \times \X^m$. On the other hand, in Assumption \hyperref[ass:C]{(C)}, if $\F_0=\sigma(U_0)$ is nonatomic, then an application of Theorem \ref{th:density-discrete} directly gives us a sequence $H^n=(H^n_k)_{k=0}^{m-1}$ of $(\F_{t_k})_{k=0}^{m-1}$-adapted processes such that \eqref{pf:density-continuous1} holds. In either case,  we may now define
\begin{align*}
X^n_t &= \frac{t-t_k}{t_{k+1}-t_k}H^n_k+ \frac{t_{k+1}-t}{t_{k+1}-t_k}H^n_{(k-1)^+}, \quad \text{ for } t \in [t_k,t_{k+1}], \ k=0,\ldots,m-1, \\
X^n_t &= H^n_{m-1}, \quad \text{ for } t \ge t_m,
\end{align*}
and conclude that $( X^n,Z) \Rightarrow ( X,Z)$ in $C(\R_+;\X) \times \Z$. Note that in the construction of $H^n=(H^n_k)_{k=1}^{m-1}$, Theorem \ref{th:density-discrete} allows us to take $H^n_k = f^n_k(U_{t_k})$ for some continuous functions $f^n_k$ on $\U_{t_k}$. Hence, we have proven that (ii) implies (iv).

With the proof that (ii) implies (iv) now complete, we complete the chain of implications by showing that (iii) implies (i), arguing  as in the proof of Theorem \ref{th:density-discrete}. Note that (i) holds if and only if for every $t$ and every bounded measurable functions $f$, $h$, and $g$ on $C([0,t];\X)$, $D([0,t];\Y)$, and $\Z$, respectively, we have 
\begin{align}
\E[f(X|_{[0,t]})h(U_t)(g(Z) - \tilde{g}(U_t))] &= 0, \label{pf:density-continuous2} \\
\text{ where } \tilde{g}(U_t) &= \E[g(Z) | \F_t]. \nonumber
\end{align}
By a routine approximation, it is enough that this holds only for \emph{continuous} $f$. Hence, if a sequence of continuous processes $X^n$ satisfies \eqref{pf:density-continuous2} as well as $(X^n,Z) \Rightarrow (X,Z)$ for some process $X$, then using Lemma \ref{le:stable-convergence} we conclude that $X$ must satisfy \eqref{pf:density-continuous2} as well. To conclude that (iii) implies (i), simply notice that any $\FF$-adapted process $X$ trivially satisfies \eqref{pf:density-continuous2}.
\end{proof}

We close this section with a topological formulation of Theorem \ref{th:density-continuous}, in the case described in Example \ref{ex:Ycadlag} with $\Z=D(\R_+;\Y)$ and $Z=Y$. On the space $D(\R_+;\Y) \times C(\R_+;\Y)$, let $Y=(Y_t)_{t \ge 0}$ and $X=(X_t)_{t \ge 0}$ denote the canonical processes, and let $\FF^Y$ and $\FF^X$ denote their natural filtrations. Fix $\nu \in \P(D(\R_+;\Y))$, and for each $t \ge 0$ let $\nu_t \in \P(\Y)$ denote the time-$t$ marginal law.
Let $\Pi^c(\nu)$ denote the set of compatible joint laws, i.e., the set of probability measures $P$ on $D(\R_+;\Y) \times C(\R_+;\Y)$ with first marginal $\nu$ and under which $\F^X_t$ is conditionally independent of $\F^Y_\infty$ given $F^Y_t$, for every $t \ge 0$. Let $\Pi^c_0(\nu)$ denote the subset of adapted joint laws, i.e., the set of probability measures of the form
\[
\nu(dy)\delta_{\hat{x}(y)}(dx),
\]
where $\hat{x} : D(\R_+;\Y) \rightarrow C(\R_+;\Y)$ is \emph{adapted} in the sense that $\hat{x}^{-1}(C) \in \F^Y_t$ for every $t \ge 0$ and every set $C \in \F^X_t$.

Notice that $P \in \P(D(\R_+;\Y) \times C(\R_+;\X))$ belongs to $\Pi^c(\nu)$ if and only if $P \circ Y^{-1} = \nu$ and the following holds: For each $t \ge 0$, each bounded continuous function $f : C([0,t];\X) \rightarrow \R$, and each bounded measurable functions $h$ and $g$ on $D([0,t];\Y)$ and $D(\R_+;\Y)$, respectively, we have
\begin{align*}
\E^P\left[f(X_{\cdot \wedge t})h(Y_{\cdot \wedge t})\left(g(Y)-\tilde{g}(Y_{\cdot \wedge t})\right)\right] &= 0, \\
\text{where } \ \tilde{g}(Y_{\cdot \wedge t}) &= \E^\nu[g(Y) \, | \, \F^Y_t].
\end{align*}
It follows from Lemma \ref{le:stable-convergence} that $\Pi^c(\nu)$ closed. For each $x \in \X$, define
\[
\P_x = \left\{P \in \P(D(\R_+;\Y) \times C(\R_+;\X)) : P(X_0=x)=1\right\}.
\]
We then deduce the following facts from Theorem \ref{th:density-continuous}:
\begin{enumerate}[(a)]
\item If $\nu_0$ is nonatomic, then the closure of $\Pi^c_0(\nu)$ is $\Pi^c(\nu)$.
\item Suppose $\nu_t$ is nonatomic for each $t > 0$, and $\nu \circ Y_0^{-1} = \delta_{y_0}$ for some $y_0 \in \Y$. Then, for each $x \in \X$, the closure of $\P_x \cap \Pi^c_0(\nu)$ is precisely $\P_x \cap \Pi^c(\nu)$.
\end{enumerate}

\begin{remark}
Note the piecewise-linear approximations in Theorem \ref{th:density-continuous} require the space $\X$ to be a convex subset of a vector space. This requirement will not be needed in the following two sections, where we approximate $X$ by piecewise-constant processes.
\end{remark}

\subsection{Continuous time, c\`adl\`ag paths}
An analogue of Theorem \ref{th:density-continuous} holds when the processes $X$ is merely c\`adl\`ag, but there is a subtle breakdown in the implication (iii) $\Rightarrow$ (i). To clarify the correct analogue in the c\`adl\`ag case requires a more careful choice of filtrations. For any filtration $\FF=(\F_t)_{t \ge 0}$, let $\FF_+ = (\F_{t+})_{t \ge 0}$ denote the right-continuous version, defined by $\F_{t+} = \cap_{s > t}\F_s$.
A filtration $\FF= (\F_t)_{t \ge 0}$ is \emph{complete} if $\F_0$ contains all of the null sets of $\F_\infty$. A generic filtration $\FF$ can of course be rendered complete by appending to $\F_0$ all of the $\F_\infty$-null sets, and the resulting filtration, denoted $\overline{\FF}$, is called the \emph{completion of $\FF$}. It is well known that $\overline{\FF}_+ = \overline{(\FF_+)}$; that is, the operations of ``completion'' and ``making right-continuous'' commute.

\begin{theorem}[C\`adl\`ag paths] \label{th:density-cadlag}
Suppose Assumptions \hyperref[ass:A]{(A)}, \hyperref[ass:B]{(B)}, and \hyperref[ass:C]{(C)} hold, and assume the process $X$ of Assumption \hyperref[ass:A]{(A.3)} has c\`adl\`ag paths. 
Consider the following statements:
\begin{enumerate}[(i)]
\item For each $t \ge 0$, $\F^X_t$ is conditionally independent of $Z$ given $\F_t$.
\item There exists a dense set $\T \subset \R_+$ such that, for each $t \in \T$,  $\F^X_t$ is conditionally independent of $Z$ given $\F_t$.
\item There exists a sequence $X^n$ of c\`adl\`ag $\FF$-adapted processes such that $(X^n,Z) \Rightarrow (X,Z)$ in $D(\R_+;\X) \times \Z$.
\item For every $t \ge 0$, $\F^X_{t+}$ is conditionally independent of $Z$ given $\F_{t+}$.
\end{enumerate}
Then $(i) \Rightarrow (ii) \Leftrightarrow (iii) \Rightarrow (iv)$. If the completion of the filtration $\FF$ is right-continuous, then (i-v) are equivalent. In particular, this holds if $\FF$ is generated by a Feller process.
Lastly, if $\X$ is homeomorphic to  a convex subset of a locally convex space, then (ii) and (iii) are equivalent to the following:
\begin{enumerate}
\item[(v)] Statement (iii) holds with $X^n$ piecewise constant between some time points $0=t_0 < t_1 < \cdots < t_n$, and constant after time $t_n$, with $X^n_{t_k} = f_k(U_{t_{k}})$ for some continuous function $f_k : \U_{t_{k}} \rightarrow \X$, for each $k=1,\ldots,n$.
\end{enumerate}
\end{theorem}

The proof makes use of a simple lemma which will be useful again in Section \ref{se:randomized-stopping}:

\begin{lemma} \label{le:compatibility-rightcontinuous}
Suppose $\FF$ and $\GG$ are two filtrations, and suppose $\H$ is a $\sigma$-field. Suppose $\F_t$ is conditionally independent of $\H$ given $\G_t$, for every $t \in \T$, where $\T \subset \R_+$ is dense. Then $\F_{t+}$ is conditionally independent of $\H$ given $\G_{t+}$, for every $t \in \R_+$.
\end{lemma}
\begin{proof}
Fix $t \ge 0$, and find $t_n \in \T$ such that $t_n > t$ and $t_n \downarrow t$. Let $A \in \H$ and $B \in \F_{t+}$. Then $B \in \F_{t_n}$ for all $n$, so $\PP(B \, | \, \G_{t_n})\PP(A \, | \, \G_{t_n}) =  \PP(B \cap A \, | \, \G_{t_n})$.
By backward martingale convergence, letting $t_n \downarrow t$ yields $\PP(B \, | \, \G_{t+})\PP(A \, | \, \G_{t+}) =  \PP(B \cap A \, | \, \G_{t+})$.
\end{proof}

\begin{proof}[Proof of Theorem \ref{th:density-cadlag}]
Clearly (i) implies (ii) and (v) implies (iii). It follows from Lemma \ref{le:compatibility-rightcontinuous} that (ii) implies (iv).

We begin with the proof that (ii) implies (iv) as well as (v) under the assumption that $\X$ is convex. Approximating by piecewise constant processes, we may assume $X$ is of the form
\[
X_t = \sum_{k=0}^mH_k1_{[t_k,t_{k+1})}(t),
\]
where $0 = t_0 < t_1 < \cdots < t_m = T$ are deterministic, with $t_{m+1} :=\infty$. As in the proof of step 3 of Theorem \ref{th:density-continuous}, we can apply Theorem \ref{th:intro:density-discrete} to find a sequence $H^n=(H^n_k)_{k=0}^m$ of $(\F_{t_k})_{k=0}^m$-adapted processes such that
\[
\left(Z,H^n_0,\ldots,H^n_m\right) \Rightarrow \left(Z,H_0,\ldots,H_m\right),
\]
in $\Z \times \X^{m+1}$. Define
\[
X^n_t = \sum_{k=0}^mH^n_k1_{[t_k,t_{k+1})}(t),
\]
and conclude that $(X^n,Z) \Rightarrow (X,Z)$ in $D(\R_+;\X) \times \Z$. Under the assumption that $\X$ is convex, we may apply Theorem \ref{th:density-discrete} instead of Theorem \ref{th:intro:density-discrete} to get $H^n_k$ of the form $H^n_k = f^n_k(U_{t_k})$ for some continuous function $f^n_k : \U_{t_k} \rightarrow \X$

We next prove that (iii) implies (ii).
For any $\FF$-adapted process $X$, it holds that
\begin{align}
\E[f(X_{\cdot \wedge t})h(Y_{\cdot \wedge t})(g(Z) - \tilde{g}(Y_{\cdot \wedge t}))] &= 0, \label{pf:density-cadlag1} \\
\text{ where } \tilde{g}(Y_{\cdot \wedge t}) &= \E[g(Z) | \F_t]. \nonumber
\end{align}
for every $t$ and every bounded functions $f$, $h$, and $g$ on $D(\R_+;\X)$, $D(\R_+;\Y)$, and $\Z$, respectively.
Suppose we are given a sequence of continuous processes $X^n$ such that $(X^n,Z) \Rightarrow (X,Z)$ and also \eqref{pf:density-cadlag1} holds (for every $t$ and every $f$, $h$, $g$). At this point, we are faced with the annoying fact that the restriction map
\[
D(\R_+;\X) \ni x \mapsto x|_{[0,t]} \in D([0,t];\X)
\]
is continuous at a point $x \in D(\R_+;\X)$ if and only if $x$ is continuous at $t$. To get around this, note that $\T = \{t \ge 0 : \PP(X_t = X_{t-})=1\}$ is dense in $\R_+$ (see \cite[Lemma 3.7.7]{ethier-kurtz}).  Hence, for $t \in \T$, and for $f$, $h$, and $g$ as above, we use Lemma \ref{le:stable-convergence} to conclude
\begin{align*}
0 &= \lim_{n\rightarrow\infty}\E[f(X^n_{\cdot \wedge t})h(U_t)(g(Z) - \tilde{g}(U_t))] \\
	&= \E[f(X_{\cdot \wedge t})h(U_t)(g(Z) - \tilde{g}(U_t))].
\end{align*}
This proves (ii).

We lastly prove that (iv) implies (i) under the additional assumption that the completion of $\FF$ is right-continuous.
Since $\F^X_t \subset \F^X_{t+}$, (iv) implies that $\F^X_t$ is conditionally independent of $Z$ given $\F_{t+}$, for every $t \ge 0$. But every set of $\F_{t+}$ differs from a set in $\F_t$ by an $\F_\infty$-null set, and we deduce easily that $\F^X_t$ is conditionally independent of $Z$ given $\F_t$, for every $t \ge 0$. Finally, it is well known that the completed natural filtration of a Feller process is automatically right-continuous; see \cite[Theorem I.47]{protter-book}.
\end{proof}

\begin{remark}
Recalling the setting of Example \ref{ex:Ycadlag}, by taking $\Z = D(\R_+;\Y)$  and $Z=Y$ we can deduce from Theorem \ref{th:density-cadlag} an analogue of  Theorem \ref{th:intro:density-continuous} in which $X$ is c\`adl\`ag.
\end{remark}

A topological formulation of Theorem \ref{th:density-cadlag} (and Theorem \ref{th:density-measurable} in the following section) is possible, along the lines of the discussions at the end of Section \ref{se:continuoustime}. We do not bother to spell this out here, as it differs only in notation from the aforementioned discussion.

\subsection{Continuous time, measurable paths}
We state one more alternative of Theorems \ref{th:density-continuous} and \ref{th:density-cadlag}, in the case that $X$ is not continuous or even c\`adl\`ag but merely measurable.
For a Polish space $\X$, let $M(\R_+;\X)$ denote the set of equivalence classes of a.e.\ equal measurable functions from $\R_+$ to $\X$. Endow $M(\R_+;\X)$ with the topology of convergence in measure, with $\R_+$ equipped with the measure $e^{-x}dx$ (any finite measure equivalent to Lebesgue measure would do). Note that $M(\R_+;\X)$ is a Polish space, and $x^n \rightarrow x$ in $M(\R_+;\X)$ if and only if $x^n|_{[0,t]} \rightarrow x|_{[0,t]}$ in Lebesgue measure for each $t > 0$.

A measurable $\X$-valued process can always be naturally identified with an $M(\R_+;\X)$-valued random variable, but we must be careful to define its natural filtration in a way that reflects the topological setting. For each $t > 0$, we may view the restriction $X|_{[0,t]}$ as a $M([0,t];\X)$-valued random variable, and we let $^*\!\F^X_t = \sigma(X|_{[0,t]}) = \sigma(\{X|_{[0,t]} \in A\} : A \subset M([0,t];\X) \text{ Borel})$. Equivalently, 
we may write
\[
^*\!\F^X_t = \sigma\left(\int_0^th(s,X_s)ds : \text{ for bounded continuous } h : [0,T] \times \X \rightarrow \R\right).
\]
A priori, this could be smaller than the $\sigma$-field $\F^X_t=\sigma(X_s : s \le t)$, because passing to the equivalence class has lost us some ``pointwise'' information.
This filtration $^*\FF^X=(^*\!\F^X_t)_{t \ge 0}$ seems to be the appropriate one to use, as evidenced by Theorem \ref{th:density-measurable} below.

\begin{remark}
The two natural filtrations $^*\FF^X$ and $\FF^X$ never disagree by too much, in the sense that one one can always find a modification of $X$ for which they agree. More precisely, let $\GG=(\G_t)_{t \ge 0}$ denote the natural filtration on $M(\R_+;\X)$ defined by letting $\G_t$ denote the $\sigma$-field generated by the restriction map $M(\R_+;\X) \ni x \mapsto x|_{[0,t]} \in M([0,t];\X)$. Then there exists a measurable map $\widehat{x} : \R_+ \times M(\R_+;\X) \rightarrow \X$ which is predictable with respect to $\GG$ and which satisfies $\hat{x}(t,x)=x_t$ for a.e.\ $t$, for each $x$.
\end{remark}

\begin{theorem}[Measurable paths] \label{th:density-measurable}
Suppose Assumptions \hyperref[ass:A]{(A)} and \hyperref[ass:B]{(B)} hold. Then the following are equivalent:
\begin{enumerate}[(i)]
\item For every $t \ge 0$, $^*\!\F^X_t$ is conditionally independent of $Z$ given $\F_t$.
\item There exists a dense set $\T \subset \R_+$ such that, for each $t \in \T$,  $^*\!\F^X_t$ is conditionally independent of $Z$ given $\F_t$.
\item There exists a sequence $X^n$ of measurable $\FF$-adapted processes such that $(X^n,Z) \Rightarrow (X,Z)$ in $M(\R_+;\X) \times \Z$.
\end{enumerate}
If $\X$ is homeomorphic to  a convex subset of a locally convex space, then (i--iii) are equivalent to the following:
\begin{enumerate}
\item[(iv)] Statement (iii) holds with $X^n$ piecewise constant between some time points $0=t_0 < t_1 < \cdots < t_n$, and constant after time $t_n$, with $X^n_{t_k} = f_k(U_{t_{k-1}})$ for some continuous function $f_k : \U_{t_{k-1}} \rightarrow \X$, for each $k=1,\ldots,n$.
\end{enumerate}
\end{theorem}
\begin{proof}
Clearly (i) implies (ii) and (iv) implies (iii).
We begin by proving that (ii) implies (iii), or (iv) under the additional assumption. The proof is very similar to that of Theorems \ref{th:density-continuous} and \ref{th:density-cadlag}, so we mention only the different steps. By an approximation, we can reduce to the case where $X$ is piecewise constant. Our generic piecewise constant process $X$ can now be assumed to be of the form
\[
X_t = H_01_{\{0\}}(t) + \sum_{k=0}^{m-1}H_k1_{(t_{k},t_{k+1}]}(t),
\]
where $0 = t_0 < t_1 < \cdots < t_m$ are deterministic, with $t_i \in \T$ for each $i$. Note that we choose $X$ to be left-continuous so that it is predictable. Now, because the topology of $M(\R_+;\X)$ is not sensitive to pointwise changes, we may take $H_0$ to be deterministic; that is, we may assume $X_t$ is constant and deterministic on a neighborhood of the origin.
For this reason, we do not need Assumption \hyperref[ass:C]{(C)}, and we apply Theorem \ref{th:intro:density-discrete} as in the previous two proofs to find a sequence $H^n=(H^n_k)_{k=1}^{m-1}$ of $(\F_{t_k})_{k=1}^{m-1}$-adapted processes such that
\[
\left(Z,H^n_1,\ldots,H^n_{m-1}\right) \Rightarrow \left(Z,H_1,\ldots,H_{m-1}\right),
\]
in $\Z \times \X^{m-1}$. Set $H^n_0=H_0$ and define
\[
X^n_t = H_01_{\{0\}}(t) + \sum_{k=0}^{m-1}H^n_k1_{(t_{k},t_{k+1}]}(t).
\]
Conclude that $(X^n,Z) \Rightarrow (X,Z)$ in $M(\R_+;\X) \times \Z$.

The proof that (iii) implies (i) is exactly the same as the last paragraph of the proof of Theorem \ref{th:density-continuous}. Note that the restriction map $x \mapsto x|_{[0,t]}$ is continuous from $M(\R_+;\X)$ to $M([0,t];\X)$, so we run into none problems that appeared in the analogous step of Theorem \ref{th:density-cadlag}.
\end{proof}

\subsection{On the case of two fixed marginals} \label{se:2marginals-dynamic}

Comparing Theorems \ref{th:intro:density-discrete} and \ref{th:intro:density-continuous} in discrete and continuous time, respectively, it is natural to wonder if there is an analogue of the latter result in which the marginal law of the process $X$ is fixed. More precisely, in Theorem \ref{th:intro:density-continuous}, does the equivalence between (i) and (ii) still hold if we require in (ii) that $X^n \stackrel{d}{=} X$ for each $n$?
Such a result would be useful in the theory of causal optimal transport, to show in particular that the Monge and Kantorovich formulations of continuous-time causal transport problems are equal, as we discussed in discrete time in Proposition \ref{pr:causal-transport-attained}.
But this turns out to be impossible at this level of generality.

For example, suppose in the context of Theorem \ref{th:intro:density-continuous} that $Y$ is a standard one-dimensional Brownian motion, and $X$ is the process given by $X_t=tZ$ where $Z$ is some non-degenerate random variable. Note that condition (a) of Theorem \ref{th:intro:density-continuous} holds in this case. Then $X$ and $Y$ cannot be coupled in such a way that $X$ is adapted with respect to the (completion of the) filtration generated by $Y$. In other words, the set of causal Monge couplings is \emph{empty}. To see why, simply note that the $\sigma$-field $\F^Y_{0+}:= \cap_{t > 0}\F^Y_t$ is trivial  by Blumenthal's zero-one law, whereas $\F^X_{0+}=\sigma(Z)$ is non-trivial.

The same idea leads to a more general class of examples: The set of causal Monge couplings is empty whenever $Y$ is a strong Markov process (whose filtration thus satisfies Blumenthal's zero-one law) and $X$ is a process for which $\F^X_{0+}$ is non-trivial.

The fundamental difficulty stems, it seems, from the wide variety of different filtrations which exist in continuous time, compared to the static case or to finite discrete time.
The static result, Proposition \ref{pr:density-two-marginals}, works in part because all nonatomic Polish probability spaces are Borel isomorphic.
Analogously, the ``most general" finite discrete-time filtrations are those generated by a sequence of independent non-atomic random variables, as indicated by the transfer principle of Lemma \ref{le:transfer-dynamic} below. Continuous-time filtrations, on the other hand, are not so neatly classified.  See \cite{tsirelson1997triple,dubins1996decreasing,emery2001vershik} for references on the rich and mysterious structural theory of filtrations, including plenty of non-trivial examples of non-isomorphic filtrations.

There is, however, one remarkable positive result of this nature in continuous time, due to \'Emery \cite{emery2005certain}. Proposition 2 of \cite{emery2005certain} shows the following: If $X$ and $Y$ are $\FF$-Brownian motions of the same dimension on some filtered probability space $(\Omega,\F,\FF,\PP)$, then there exists a sequence $(X^n)$ of Brownian motions such that $\FF^{X^n}=\FF^Y$ up to null sets for each $n$ and $(X^n,Y) \Rightarrow (X,Y)$. In other words, any joint Brownian motions can be approximated weakly in law by mutually adapted Brownian motions. The proof in \cite{emery2005certain} does not seem to generalize, as it relies heavily on the characterization of a centered multivariate Gaussian distribution by its covariance matrix.

\subsection{On the notion of compatibility} \label{se:general-compatibility}

This section has made frequent mention of a \emph{compatibility} property, which we discuss here in more generality. There is an ever-growing list of interesting equivalent conditions, some of which are summarized in the following:

\begin{theorem}[Theorem 3 of \cite{bremaudyor-changesoffiltrations}, Theorem 2 of \cite{aksamit2016projections}] \label{th:compatible}
On some probability space, consider two filtrations $\GG$ and $\FF$ with $\FF \subset \GG$. The following are equivalent:
\begin{enumerate}[(i)]
\item $\G_t$ is conditionally independent of $\F_\infty$ given $\F_t$, for every $t$.
\item Every bounded $\FF$-martingale is a $\GG$-martingale.
\item Every $\GG$-stopping time $\tau$ is a \emph{$\FF$-pseudo stopping time}, meaning $\E[M_\tau]=\E[M_0]$ for every uniformly integrable $\FF$-martingale.
\item For every $t$ and every integrable $\F_\infty$-measurable  $X$, $\E[X | \F_t] = \E[X | \G_t]$ a.s.
\item For every $t$ and every integrable $\G_t$-measurable  $X$, $\E[X | \F_t] = \E[X | \F_\infty]$ a.s.
\end{enumerate}
\end{theorem}

The last two statements are easily seen to be restatements of (i). 
The condition (ii) is known as the \emph{H-hypothesis} in the filtering literature \cite{bremaudyor-changesoffiltrations}, or often as the \emph{immersion property} in the literature on progressive enlargements of filtration \cite{jeanblanc2009immersion}, whereas (iii) is a recent result of \cite{aksamit2016projections}.

There is a natural additional characterization available in the case where $\FF=\FF^Y$ is the (unaugmented) filtration generated by a c\`adl\`ag process $Y$ with independent increments and taking values in, say, a separable Fr\'echet space. In this case, the conditions (i-v) are equivalent to the following:
\begin{enumerate}
\item[(vi)] The increments of $Y$ are independent with respect to the larger filtration $\GG$. That is, $(Y_t-Y_s)_{t \ge s}$ is independent of $\G_s$ for every $s \ge 0$.
\end{enumerate}
This is most easily checked to be equivalent to property (v), by using the fact that $\F_\infty$ decomposes as $\F_\infty = \F_t \vee \sigma(Y_s - Y_t : s \ge t)$ for each $t \ge 0$.

In a sense, the notion of compatibility appeared in the work of Yamada and Watanabe \cite{yamada-watanabe}, which illustrates how it arises naturally with weak solutions of stochastic differential equations. In this context, one simply needs the driving Brownian motion to remain Brownian relative to a larger filtration which includes the solution process.
Similarly, the work of Jacod and M{\'e}min \cite{jacodmemin-weaksolution} on semimartingale-driven stochastic differential equations crucially uses compatibility, which they referred to as \emph{very good extension}. See Kurtz \cite{kurtz-yw2013} for a recent elaboration of this theme. In the study of existence of (stochastic) optimal controls, El Karoui et al. \cite{elkaroui-compactification} make use of this notion under the name of \emph{natural extension}.

\section{Randomized stopping times} \label{se:randomized-stopping}

We next turn to analogues of Theorems \ref{th:density-continuous} and \ref{th:density-cadlag} in which the process $X$  is replaced with a stopping time.
Given a random time $\tau$ (i.e., a $[0,\infty]$-valued random variable), its natural filtration $\FF^\tau=(\F^\tau_t)_{t \ge 0}$ is defined by
\[
\F^\tau_t = \cap_{s > t}\sigma(\tau \wedge s) = \sigma(\{\tau \le s\} : s \le t).
\]
Note that $\FF^\tau$ is right-continuous, and also that $\F^\tau_\infty = \sigma(\tau)$. Given a filtration $\FF=(\F_t)_{t \ge 0}$, we say that a random time $\tau$ is an $\FF$-stopping time if, as usual, $\{\tau \le t\} \in \F_t$ for every $t$. Given also a random variable $Z$ with $\F_t \subset \sigma(Z)$ for all $t \ge 0$, we say that a random time $\tau$ is an \emph{$(\FF,Z)$-randomized stopping time} if $\F^\tau_t$ is conditionally independent of $Z$ given $\F_t$, for every $t \ge 0$. The following are easily seen to be equivalent:
\begin{enumerate}
\item $\tau$ is an $(\FF,Z)$-randomized stopping time.
\item $\PP(\tau \le t | Z) = \PP(\tau \le t | \F_t)$ a.s., for every $t \ge 0$.
\end{enumerate}
As an important special case, suppose that the random variable $Z$ satisfies $\sigma(Z) = \F_\infty$. The we shorten ``$(\FF,Z)$-randomized stopping time'' to ``$\FF$-randomized stopping time,'' meaning one of the following equivalent conditions holds:
\begin{enumerate}
\item $\F^\tau_t$ is conditionally independent of $\F_\infty$ given $\F_t$, for every $t \ge 0$.
\item $\PP(\tau \le t | \F_\infty) = \PP(\tau \le t | \F_t)$ a.s., for every $t \ge 0$.
\item Every $\FF$-martingale is an $\FF \vee \FF^\tau$-martingale.
\end{enumerate}

The following extends \cite[Theorem 6.4]{carmonadelaruelacker-mfgoftiming}:

\begin{theorem} \label{th:density-stopping}
Under Assumptions \hyperref[ass:A]{(A)}(1-2) and \hyperref[ass:B]{(B)}, consider the following statements:
\begin{enumerate}[(i)]
\item $\tau$ is an $(\FF,Z)$-randomized stopping time.
\item There exists a sequence of $\FF$-stopping times $\tau_n$ such that $(Z,\tau_n) \Rightarrow (Z,\tau)$ in $\Z \times [0,\infty]$.
\item There exists a sequence of $\FF_+$-stopping times $\tau_n$ such that $(Z,\tau_n) \Rightarrow (Z,\tau)$ in $\Z \times [0,\infty]$.
\item $\tau$ is an $(\FF_+,Z)$-randomized stopping time.
\end{enumerate}
Then the implications $(i) \Rightarrow (ii) \Leftrightarrow (iii) \Leftrightarrow (iv)$ hold.
If the completion of $\FF$ is right-continuous, then (i-iv) are equivalent.
\end{theorem}

\begin{remark}
Recalling the setting of Example \ref{ex:Ycadlag}, we deduce Theorem \ref{th:intro:density-stopping} from Theorem \ref{th:density-stopping} by taking $\Z = D(\R_+;\Y)$  and $Z=Y$.
\end{remark}

The proof is somewhat involved, so we prepare with a series of lemmas.

\begin{lemma} \label{le:continuousae}
For every bounded $\F^\tau_t$-measurable random variable $H$, there exits a sequence of uniformly bounded functions $g_n : [0,\infty] \rightarrow \R$ such that $g_n(\tau) \rightarrow H$ in $L^1$, $g_n(\tau)$ is $\F^\tau_t$-measurable, and $g_n$ is continuous at every point but $t$.
\end{lemma}
\begin{proof}
First notice that being $\F^\tau_t$-measurable, $H$ is necessarily of the form:
\[
H = h(\tau)1_{[0,t]}(\tau) + c1_{(t,\infty]}(\tau)
\]
for some bounded measurable function $h : [0,t] \rightarrow \R$ and some $c \in \R$.  Now simply approximate $h$ in $L^1[0,t]$ by continuous functions $h_n : [0,t] \rightarrow \R$, and define
\[
g_n(s) = h_n(s)1_{[0,t]}(s) + c1_{(t,\infty]}(s).
\]
\end{proof}

\begin{lemma} \label{le:time-map}
Define a map $\Phi : D(\R_+;\{0,1\}) \rightarrow [0,\infty]$ by
\[
\Phi(h) = \inf\{t \ge 0 : h(t) =1\}.
\]
Then $\Phi$ is continuous at every point $h$ of the form $h(\cdot) = 1_{[s,\infty)}(\cdot)$, where $s \ge 0$.
\end{lemma}
\begin{proof}
First assume $h(t) = 1_{[s,\infty)}(t)$ for some fixed $s \ge 0$, and let $h^n \in D(\R_+;\{0,1\})$ with $h^n \rightarrow h$. Note that $\Phi(h)=s$. It is straightforward to check that $\{\Phi(h^n) : n \ge 1\}$ is bounded. Suppose, along a subsequence, that $\Phi(h^n) \rightarrow t$, for some $t \ge 0$. According to \cite[Proposition 3.6.5]{ethier-kurtz}
the set $\{h^n(\Phi(h^n)) : n \ge 1\}$ is precompact, and the set of limit points is contained in $\{h(t-), h(t)\}$. As $h(t) \ge h(t-)$ and $h^n(\Phi(h^n)) =1$, we must have $h(t)=1$, or equivalently $t \ge s$.  Similarly, for each $\epsilon > 0$, the set $\{h^n(\Phi(h^n) - \epsilon) : n \ge 1\}$ is precompact, and the set of limit points is contained in $\{h((t-\epsilon)-), h(t-\epsilon)\}$. But $h^n(\Phi(h^n)-\epsilon) =0$, and we conclude that $h((t-\epsilon)-) = 0$ or equivalently $s \ge t-\epsilon$. Thus $\Phi(h^n) \rightarrow s = \Phi(h)$.
\end{proof}

As a final preparation, we show how any $(\FF_+,Z)$-randomized stopping time can be approximated in a suitable sense by $(\FF,Z)$-randomized stopping times. First, we recall the well known non-randomized analogue:

\begin{lemma}[Corollary IV.58 of \cite{dellacherie-meyer}] \label{le:nonrandomized-approximation}
Let $\FF$ be any filtration, and let $\tau$ be an $\FF_+$-stopping time. Then there exists a sequence of $\FF$-stopping times such that $\tau_n \downarrow \tau$ a.s.
\end{lemma}
\begin{proof}
Set $\tau_n = \tau + 1/n$.
\end{proof}

\begin{lemma} \label{le:randomized-approximation}
Let $\tau$ be an $(\FF_+,Z)$-randomized stopping time. There exists a sequence of $(\FF,Z)$-randomized stopping times $\tau_n$ such that $\tau_n \to \tau$ a.s.
\end{lemma}
\begin{proof}
Define $\tau_n = \tau + 1/n$. To see that $\tau_n$ is an $(\FF,Z)$-randomized stopping time, simply note that
\begin{align*}
\PP(\tau_n \le t | Z) &= \PP(\tau \le t -1/n| Z) = \PP(\tau \le t - 1/n | \F_{(t-1/n)^++}) \\
	&= \PP(\tau_n \le t | \F_{(t-1/n)^++}).
\end{align*}
Because $\F_{(t-1/n)^++} \subset \F_t \subset \sigma(Z)$, this implies $\PP(\tau_n \le t | Z) = \PP(\tau_n \le t | \F_t)$.
\end{proof}

\subsection*{Proof of Theorem \ref{th:density-stopping}}
Clearly (ii) implies (iii), because every $\FF$-stopping time is an $\FF_+$-stopping time. Moreover, every $\FF_+$ stopping time can be written as the decreasing (almost sure) limit of a sequence of $\FF$-stopping times (see Lemma \ref{le:nonrandomized-approximation}). Hence, (iii) implies (ii). To prove (iv) implies (iii), we can use Lemma \ref{le:randomized-approximation} to approximate our $(\FF_+,Z)$-randomized stopping time by $(\FF,Z)$-randomized stopping times, and we can then use the fact that (i) implies (iii), which we prove next.

To show that (i) implies (iii), fix an $(\FF,Z)$-randomized stopping time $\tau$. By Lemma \ref{le:randomized-approximation}, we may approximate $\tau$ a.s.\ by $(\FF,Z)$-randomized stopping times which almost surely do not take the value zero. Approximating then by $\tau \wedge n$, we may further assume that $\tau < \infty$ a.s. Hence, without loss of generality, we assume $\PP(\tau =0) = \PP(\tau =\infty)=0$. Define $H_t = 1_{\{\tau \le t\}}$, and note that $H$ is a c\`adl\`ag $\{0,1\}$-valued process with $\PP(H_0=0)=1$. Moreover, $\F^H_t = \F^\tau_t$ is conditionally independent of $\F_\infty$ given $\F_t$, for every $t \ge 0$.
We may now apply Theorem \ref{th:density-cadlag} to find a sequence of c\`adl\`ag $\FF$-adapted $\{0,1\}$-valued processes $H^n$ such that $(Z,H^n) \Rightarrow (Z,H)$ in $\Z \times D(\R_+;\{0,1\})$. Let $\tau_n = \inf\{t \ge 0 : H^n_t=1\}$. Then $\tau_n$ is an $\FF$-stopping time because $H^n$ is right-continuous and $\FF$-adapted. By Lemma \ref{le:time-map}, $H$ almost surely belongs to the set of continuity points of $\Phi$, and we conclude using Lemma \ref{le:stable-convergence} that
\[
(Z,\tau_n) = (Z,\Phi(H^n)) \Rightarrow (Z,\Phi(H)) = (Z,\tau), \text{ in } \Z \times [0,\infty].
\]

We next show that (iii) implies (iv). Let $\T = \{t \ge 0 : \PP(\tau = t) = 0\}$, and note that $\T$ is dense in $\R_+$. Fix $t \in \T$, and let $f_t=f_t(\tau)$ be a bounded $\F^\tau_t$-measurable random variable which is continuous at every point but $t$. Let $h_t(Z)$ be bounded and $\F_{t+}$-measurable, and let $g(Z)$ be bounded and $Z$-measurable. Because $\tau_n$ is an $\FF_+$ stopping time, for each $n$, we have
\begin{align*}
0 &= \E\left[f_t(\tau_n)h_t(Z)\left(g(Z) - \E[g(Z) | \F_{t+}]\right)\right].
\end{align*}
Use Lemma \ref{le:stable-convergence} to pass to the limit to get
\begin{align*}
0 &= \E\left[f_t(\tau)h_t(Z)\left(g(Z) - \E[g(Z) | \F_{t+}]\right)\right].
\end{align*}
This holds for every bounded $Z$-measurable $g$, for every bounded $\F_{t+}$-measurable $h_t$, and every bounded $\F^\tau_t$-measurable $f_t$ which is continuous at every point but $t$. In fact, this is valid for \emph{every} bounded measurable $\F^\tau_t$-measurable $f_t$, thanks to  Lemma \ref{le:continuousae}. We conclude that, for $t \in \T$, $\F^\tau_t$ is conditionally independent of $Z$ given $\F_{t+}$. Use Lemma \ref{le:compatibility-rightcontinuous} to complete the proof of (iv).
Finally, under the additional assumption that the completion of $\FF$ is right-continuous, it is clear that (iv) and (i) are equivalent.
\hfill\qedsymbol

\section{Extreme points} \label{se:extremepoints}

In the previous section we saw in what sense adapted processes (resp. stopping times) are dense in the set of compatible processes (resp. randomized stopping times), in a certain joint-distributional sense. In this section we study the convex structure of the associated sets of joint distributions. In fact, these same dense subsets are often precisely the extreme points.

\subsection{Transfer principles} \label{se:transfer}

We first recall a crucial lemma, often known as the \emph{transfer principle}. We will see in Sections \ref{se:extremepoints-static} and \ref{se:extremepoints-discrete} how a transfer principle, when available, immediately provides the extreme points of sets of joint distributions with one fixed marginal.

\begin{lemma}[Theorem 6.10 of \cite{kallenberg-foundations}] \label{le:transfer}
Suppose $X$ and $Y$ are random elements of measurable spaces $\X$ and $\Y$, respectively, where $\X$ is a Polish space. Then there exist a measurable function $f : \Y \times [0,1] \rightarrow \X$ and (perhaps on an extension of the probability space) a uniformly distributed random variable $U$, independent of $(X,Y)$, such that $(X,Y) \stackrel{d}{=} (f(Y,U),Y)$.
\end{lemma}

A dynamic (discrete-time) version of Lemma \ref{le:transfer} follows by induction:

\begin{lemma} \label{le:transfer-dynamic}
Suppose $X=(X_n)_{n=1}^N$ and $Y=(Y_n)_{n=1}^N$ are stochastic processes with values in measurable spaces $\X$ and $\Y$, respectively, where $\X$ is a Polish space. If $X_n$ is conditionally independent of $\F^Y_N$ given $\F^Y_n$, for every $n$, then there exist measurable functions $f_n : \Y^n \times [0,1] \rightarrow \X$ and (perhaps on an extension of the probability space) an independent uniform random variable $U$ such that, if $Z = (Z_1,\ldots,Z_N)$ where
\[
Z_n = f_n(Y_1,\ldots,Y_n,U),
\]
then $(Z,Y)  \stackrel{d}{=} (X,Y)$.
\end{lemma}
\begin{proof}
Abbreviate $X^n = (X_1,\ldots,X_n)$ for $n \in \N$ and similarly for other processes. 
By Borel isomorphism, it suffices to show that there exist measurable functions $f_n : \Y^n \times [0,1]^n \rightarrow \X$ and independent uniform random variables $(U_1,\ldots,U_N)$ such that, if $Z = (Z_1,\ldots,Z_N)$ where
\[
Z_n = f_n(Y_1,\ldots,Y_n,U_1,\ldots,U_n),
\]
then $(Y,Z)$ has the same law as $(Y,X)$, where $U_k$ is independent of $(X^k,Y^N)$ for each $k=1,\ldots,n$.
We prove this by induction, with the $N=1$ case is covered by Lemma \ref{le:transfer}. Suppose the claim is proven for some $N$. Thanks to Lemma \ref{le:transfer}, we may find an independent uniform $U_{N+1}$ and a measurable function $g : \X^N \times \Y^{N+1} \times [0,1] \rightarrow \X$ such that
\[
(X^N,Y^{N+1},g(X^N,Y^{N+1},U_{N+1})) \overset{d}{=} (X^N,Y^{N+1},X_{N+1}).
\]
We may take $U_1,\ldots,U_{N+1}$ to be independent of each other and of $Y^{N+1}$, with also $U_k$ independent of $X^k$ for each $k=1,\ldots,N+1$.
Because $X^N$ is conditionally independent of $Y^{N+1}$ given $Y^N$, we have $(X^N,Y^{N+1}) \stackrel{d}{=} (Z^N,Y^{N+1})$. Hence,
\[
(X^N,Y^{N+1},g(X^N,Y^{N+1},U_{N+1})) \overset{d}{=} (Z^N,Y^{N+1},g(Z^N,Y^{N+1},U_{N+1})).
\]
Finally, define
\begin{align*}
f_{N+1}&(y_1,\ldots,y_{N+1},u_1,\ldots,u_{N+1}) \\
	&= g\Bigl(f_1(y_1,u_1),\ldots,f_N(y_1,\ldots,y_N,u_1,\ldots,u_N),y_1,\ldots,y_{N+1},u_{N+1}\Bigr).
\end{align*}
\end{proof}

In continuous time, there is no analogous transfer principle in general, as the following example illustrates.

\begin{example} \label{ex:tsirelson}
Let $W = (W_t)_{t \ge 0}$ and $X = (X_t)_{t \ge 0}$ be continuous processes, defined on some common filtered probability space $(\Omega,\F,\FF,\PP)$, such that $X$ and $W$ are adapted to $\FF$, $W$ is an $\FF$-Wiener process, and $(W,X)$ satisfy the stochastic differential equation
\begin{align}
X_t = \int_0^tb(s,X)\, ds + W_t, \label{def:tsirelson}
\end{align}
where $b : \R_+ \times C(\R_+) \rightarrow \R$ is the function of Tsirelson's example \cite{tsirel1975example}. We know then that $\F^X_t$ is strictly larger than $\F^W_t$ for some $t \ge 0$. Because $W$ is a Wiener process with respect to $\FF$, it follows that $\F^X_t$ is conditionally independent of $\F^W_\infty$ given $\F^W_t$, for every $t \ge 0$.

Now suppose also that $X = F(W,U)$ a.s., where $U$ is uniformly distributed on $[0,1]$ and independent of $W$, and where $F : C(\R_+) \times [0,1] \rightarrow C(\R_+)$ is an adapted function. By ``adapted function" we mean that for every $(w,w',u,t) \in C(\R_+) \times C(\R_+) \times [0,1] \times \R_+$ satisfying $w_s = w'_s$ for all $s \le t$ we have $F(w,u)(t) = F(w',u)(t)$. The independence of $W$ and $U$ ensures that $W$ is a Brownian motion (in its own filtration) under the regular conditional measure $P_u := \PP(\cdot \ | \ U=u)$, for Lebesgue-a.e.\ $u \in [0,1]$. Now, under $P_u$, it holds that $X$ is adapted to the completion of $\FF^W$, that $W$ is Brownian motion under $\FF^W$, and that the SDE \eqref{def:tsirelson} holds a.s. But this provides a \emph{strong solution} of the SDE \eqref{def:tsirelson}, which is a contradiction.
\end{example}

\subsection{Extreme points in the static case} \label{se:extremepoints-static}

Here we  review how the transfer principle leads to a well known description of the extreme points of the convex set $\Pi(\mu)$ of probability measures on $\X \times \Y$ with fixed first marginal $\mu$. This is a good point to recall the definitions of $\Pi(\mu)$ and $\Pi_0(\mu)$ from \eqref{def:Pi(mu)} and \eqref{def:Pi0(mu)}, as well as $\Pi(\mu,\nu)$ and $\Pi_0(\mu,\nu)$ defined in \eqref{def:Pi(mu,nu)} and \eqref{def:Pi0(mu,nu)}.

\begin{proposition} \label{pr:extremepoints-static}
Let $\X$ and $\Y$ be Polish spaces, and let $\mu \in \P(\X)$. 
The set $\Pi(\mu)$ is convex. Moreover, every $P \in \Pi(\mu)$ can be written as $P(\cdot) = \int_{\Pi_0(\mu)}m(\cdot)M(dm)$, for some probability measure $M$ on $\Pi_0(\mu)$.
\end{proposition}
\begin{proof}
Convexity is obvious.
Fix $P \in \Pi(\mu)$. 
By Lemma \ref{le:transfer}, there exists a measurable function $f : \X \times [0,1] \rightarrow \Y$ such that $(\mu \times \mathrm{Leb}) \circ [(x,u) \mapsto (x,f(x,u))]^{-1} = P$. Define $P_u \in \Pi_0(\mu)$ by
\[
P_u(dx,dy) = \mu(dy)\delta_{f(x,u)}(dy).
\]
Then, for any bounded measurable function $\varphi$ on $\X \times \Y$, we have
\begin{align*}
\int_{\X \times \Y}\varphi(x,y)P(dx,dy) &= \int_0^1\int_{\X}\varphi(x,f(x,u))\,\mu(dx)\,du \\
	&= \int_0^1\int_{\X\times\Y}\varphi(x,y)\,P_u(dx,dy)\,du.
\end{align*}
\end{proof}

\begin{remark}
Recalling from Proposition \ref{pr:density-one-marginal} that $\Pi_0(\mu)$ is dense in $\Pi(\mu)$ when $\mu$ is nonatomic, we have encountered a closed convex set $\Pi(\mu)$ which is the closure of its extreme points. This is not as peculiar as it may at first seem, and an intriguing result of Klee \cite{klee1959some} shows that this situation 
is generic in a topological sense in infinite dimensional Banach spaces.
\end{remark}

\subsection{Extreme points in discrete time} \label{se:extremepoints-discrete}

We now extend Proposition \ref{pr:extremepoints-static}, which requires a bit of notation but follows essentially the same proof, taking the dynamic form of the transfer principle Lemma \ref{le:transfer-dynamic} for granted.

Fix $N$ throughout this section. Let $\Y$ be a measurable space and $\X$ a Polish space. Let $Y=(Y_1,\ldots,Y_N)$ denote the canonical process (identity map) on $\Y^N$, and let $\FF^Y=(\F^Y_n)_{n=1}^N$ denote its natural filtration. Similarly, let $X=(X_1,\ldots,X_N)$ denote the canonical process (identity map) on $\X^N$, and let $\FF^X=(\F^X_n)_{n=1}^N$ denote its natural filtration. Both canonical processes extend in the obvious way to $\Y^N \times \X^N$. Similarly, both filtrations extend in the natural way to $\Y^N \times \X^N$, e.g., by identifying $\F^Y_n$ with $\F^Y_n \otimes \{\emptyset,\X^N\}$. 

Fix a joint distribution $\mu$ on $\Y^N$.
Let $\Pi^c(\mu)$ denote the set of compatible joint laws, i.e., the set of probability measures $P$ on $\Y^N \times \X^N$ with first marginal $\mu$ and under which $\F^X_n$ is conditionally independent of $\F^Y_N$ given $F^Y_n$, for every $n=1,\ldots,N$. Let $\Pi^c_0(\mu)$ denote the subset of adapted joint laws, i.e., the set of probability measures on $\Y^N \times \X^N$ of the form
\[
\mu(dy)\delta_{\hat{x}(y)}(dx),
\]
where $\hat{x} = (\hat{x}_1,\ldots,\hat{x}_N) : \Y^N \rightarrow \X^N$ is \emph{adapted} in the sense that $\hat{x}_n(y_1,\ldots,y_N)$ depends only on $y_1,\ldots,y_n$, for each $n$, or equivalently $\hat{x}^{-1}(C) \in \F^Y_n$ for every set $C \in \F^X_n$.
Note that Theorem \ref{th:intro:density-discrete} says that $\Pi^c_0(\mu)$ is dense in $\Pi^c(\mu)$ when the law of $Y_1$ is nonatomic and $\X$ is convex.

The following theorem describes the convex structure of $\Pi^c(\mu)$. Taking some Choquet theory for granted, this is equivalent to the recent \cite[Theorem 6.1]{BaBeEdPi17}.

\begin{theorem} \label{th:extremepoints-discrete}
The set $\Pi^c(\mu)$ is convex. Moreover, every $P \in \Pi^c(\mu)$ can be written as $P(\cdot) = \int_{\Pi^c_0(\mu)}m(\cdot)M(dm)$, for some probability measure $M$ on $\Pi^c_0(\mu)$.
\end{theorem}
\begin{proof}
Convexity is straightforward after noticing that membership in $\Pi^c(\mu)$ is characterized by a family of linear constraints; see \eqref{pf:compat-constraint}. 
Fix $P \in \Pi^c(\mu)$. 
By Lemma \ref{le:transfer-dynamic}, there exist measurable functions $f_n : \Y^n \times [0,1] \rightarrow \X$, for $n=1,\ldots,N$, such that, if $f : \Y^N \times [0,1] \rightarrow \X^N$ is defined by
\[
f(y_1,\ldots,y_N,u) = (f_1(y_1,u),\ldots,f_N(y_1,\ldots,y_N,u)),
\]
then $(\mu \times \mathrm{Leb}) \circ [(x,u) \mapsto (x,f(x,u))]^{-1} = P$. Define $P_u$ by
\[
P_u(dy,dx) = \mu(dy)\delta_{f(x,u)}(dx),
\]
and note that the structure of $f$ ensures that $P_u \in \Pi^c_0(\mu)$. Finish as in the proof of Proposition \ref{pr:extremepoints-static}.
\end{proof}

\subsection{Extreme points in continuous time} \label{se:counterexample}

Using Example \ref{ex:tsirelson} and adapting the proof of Theorem \ref{th:extremepoints-discrete}, we can show that the natural analogue of Theorem \ref{th:extremepoints-discrete} fails in continuous time. Let us make this precise: Let $\X = C(\R_+)$. Let $\FF^i=(\F^i_t)_{t \ge 0}$ denote the two filtrations generated by the canonical processes $(X^1,X^2)$ on $\X^2$. Let $\Pi^c(\W)$ denote the set of joint laws on $\X^2$ compatible with Wiener measure. Precisely, let $\W$ denote Wiener measure on $\X$, and let $\Pi^c(\W)$ denote the set of $Q \in \P(\X^2)$ with first marginal $\W$ such that $\F^2_t$ is conditionally independent of $\F^1_\infty$ given $\F^1_t$, for every $t \ge 0$. Let $\Pi^c_0(\W)$ denote the subset of adapted joint laws, i.e., the set of probability measures on $\X^2$ of the form
\[
\W(d\omega)\delta_{F(\omega)}(dx),
\]
where $F : \X \rightarrow \X$ is a measurable and \emph{adapted} function in the sense that $F(\omega)(t)=F(\omega')(t)$ whenever $\omega(s)=\omega'(s)$ for all $s \le t$.

Next, recall the probability space $(\Omega,\F,\FF,\PP)$ and the processes $W$ and $X$ defined in Example \ref{ex:tsirelson}. Let $P = \PP \circ (W,X)^{-1}$, and note that $P \in \Pi^c(\W)$.  We claim that there is no measurable map $[0,1] \ni u \mapsto P_u \in \Pi^c_0(\W)$ such that $P(\cdot) = \int_0^1P_u(\cdot)du$. Indeed, suppose there were such a map $u \mapsto P_u$. By the disintegration theorem, we may write
\[
duP_u(d\omega,dx) = du\W(d\omega)\delta_{F(u,\omega)}(dx),
\]
for some measurable function $F: [0,1] \times \X \rightarrow \X$ with the property that $F(u,\cdot)$ must be an adapted function for each $u$. This would imply that we have an adapted map $F$ for which $\PP(X = F(U,W)) =1$. As explained in Example \ref{ex:tsirelson}, this is impossible.

\subsection{Extreme points of randomized stopping times}

Although the previous section showed that the extreme point story breaks down in continuous time, this section shows that there is no such difficulty when working with stopping times as in Section \ref{se:randomized-stopping}. The description of extreme points described in Theorem \ref{th:stopping-extremepoints} is not new (see, e.g., \cite{ghoussoub1982integral,edgar1982compactness}), but we include it for the sake of completeness.

Let $(\Omega,\F,\mu)$ be an arbitrary probability space. Let $\FF=(\F_t)_{t \ge 0}$ be a filtration on $\Omega$, and with some abuse of notation let $\FF$ denote the same filtration extended to $\overline{\Omega} := \Omega \times [0,\infty]$ by identifying $\F_t$ with $\F_t \otimes \{\emptyset,[0,\infty]\}$. Similarly, let $\tau$ denote the identity map on $[0,\infty]$, and extend $\tau$ to $\overline{\Omega}$ by setting $\tau(\omega,t) = t$. Let $\FF^\tau=(\F^\tau_t)_{t \ge 0}$ denote the filtration on $[0,\infty]$ defined by
\[
\F^\tau_t = \sigma(\{\tau \le s\} : \ s \le t),
\]
and again abuse notation by considering $\FF^\tau$ as a filtration on $\overline{\Omega}$.

Let $\Pi^c(\mu)$ denote the set of $\FF$-randomized stopping time laws, defined more precisely as the set of $P \in \P(\overline{\Omega})=\P(\Omega \times [0,\infty])$ with first marginal $\mu$ and with $\F^\tau_t$ conditionally independent of $\F_\infty$ given $\F_t$, for every $t \ge 0$. The latter constraint is equivalent to requiring $P(\tau \le t | \F_\infty) = P(\tau \le t | \F_t)$ a.s., for every $t$. Let $\Pi^c_0(\mu)$ denote the subset of $\FF$-stopping time laws, defined as
\[
\Pi^c_0(\mu) = \left\{\mu(d\omega)\delta_{\hat{\tau}(\omega)}(dt) : \hat{\tau} : \Omega \rightarrow [0,\infty] \text{ is an } \FF-\text{stopping time}\right\}.
\]
Note that Theorem \ref{th:density-stopping} showed that, if $\Omega=D(\R_+;\X)$ for some Polish space $\X$, if $\FF$ is the canonical right-continuous filtration, and if the process $X$ with law $\mu$ is such that $X_t$ is nonatomic for every $t > 0$, then the weak closure of $\Pi^c_0(\mu)$ is precisely $\Pi^c(\mu)$.

The following theorem describes the convex structure of $\Pi^c(\mu)$. Endow $\Pi^c_0(\mu)$ with the $\sigma$-field generated by the maps $\Pi^c_0(\mu) \ni P \mapsto P(C)$, for sets $C \in \F_\infty \vee \F^\tau_\infty$.

\begin{theorem} \label{th:stopping-extremepoints}
The set $\Pi^c(\mu)$ is convex. Moreover, every $P \in \Pi^c(\mu)$ can be written as $P(\cdot) = \int_{\Pi^c_0(\mu)}m(\cdot)M(dm)$ for some probability measure $M$ on $\Pi^c_0(\mu)$.
\end{theorem}
\begin{proof}
To prove convexity is straightforward, simply note that $P \in \Pi^c(\mu)$ if and only if $P(\cdot \times [0,\infty]) = \mu(\cdot)$ and
\[
\int_{\Omega\times[0,\infty]} (f(\omega) - f_t(\omega))1_{[0,t]}(s)P(d\omega,ds) = 0,
\]
for every $t \ge 0$ and every bounded $\F_\infty$-measurable function $f$ on $\Omega$, with $f_t := \E^\mu[f | \F_t]$.

Let $P \in \Pi^c(\mu)$, and disintegrate $P(d\omega,dt) = \mu(d\omega)P(\omega,dt)$. define $A : \Omega \times [0,1] \rightarrow [0,\infty]$ by
\begin{align*}
A(\omega,u) = \inf\left\{t \ge 0 : P(\omega,[0,t]) \ge u\right\}.
\end{align*}
Then $P(\omega,\cdot) = \mathrm{Leb} \circ A(\omega,\cdot)^{-1}$, and we conclude that for $t \ge 0$
\begin{align}
P(\omega,[0,t]) &= \int_0^11_{\{A(\omega,u) \le t\}}du = \int_0^1 \widehat{P}_u(\omega,[0,t])du, \label{pf:extreme-stopping1}
\end{align}
where we define $P_u \in \Pi^c_0(\mu)$ for $u \in [0,1]$ by setting $\widehat{P}_u(\omega,dt) = \delta_{A(\omega,u)}(dt)$ and $P_u(d\omega,dt) = \mu(d\omega)\widehat{P}_u(\omega,dt)$. This shows that $P(\cdot) = \int_0^1P_u(\cdot)\,du$.
\end{proof}

\section{Stochastic optimal control} \label{se:optimalcontrol}

In this section we illustrate how the notion of compatibility arises naturally in stochastic optimal control.
Consider a standard Borel probability space $(\Omega,\F,\PP)$ supporting a c\`adl\`ag process $Z$, with values in $\R^d$, which is a semimartingale in its own filtration.
Let $T > 0$ denote the time horizon, and let $A$ be a Polish space, called the \emph{action space}. We are given a function
\[
g : [0,T] \times \Omega \times \R^d \times A \rightarrow \R^{d \times m}.
\]
Assume the following:
\begin{itemize} 
\item $g=g(t,\omega,x,a)$ is jointly  continuous in $(x,a)$ and $\FF^Z$-predictable in $(t,\omega)$.
\item $g$ is Lipschitz in $x$, uniformly in $(t,\omega,a)$.
\item There exists $c > 0$ such that $|g(t,\omega,x,a)| \le c(1+|x|)$, for all $(t,\omega,x,a)$.
\end{itemize} 
For a filtration $\GG$ on $\Omega$, let $\AA(\G)$ denote the set of $\GG$-predictable $A$-valued processes.

Let $Y=(Y_t)_{t \ge 0}$ be another $\FF^Z$-adapted c\`adl\`ag process, the role of which is to determine what information is available to the agent. Assume $Y_t$ is nonatomic for every $t > 0$.
Now fix a filtration $\GG$ satisfying $\FF^Z \subset \GG$ and $\G_t \subset \F$, and also $\G_t$ is conditionally independent of $\F^Z_\infty$ given $\F^Y_t$, for every $t \ge 0$. Thanks to the implication $(i)\Rightarrow(ii)$ of Theorem \ref{th:compatible}, $Z$ is a $\GG$-semimartingale.
Hence, for every $\alpha \in \AA(\GG)$, there exists a unique (see \cite[Theorem 4.5]{jacodmemin-weaksolution}) process $X^\alpha=(X^\alpha_t)_{t \in [0,T]}$ satisfying
\[
dX^\alpha_t = g(t,X^\alpha_{t-},\alpha_t)dZ_t, \quad X^\alpha_0 = x_0,
\]
where $x_0 \in \R^d$ is a fixed initial state.

Consider the optimization problem,
\begin{align*}
\sup_{\alpha \in \AA(\FF^Y)}F\left(\L(X^\alpha,\alpha,Z)\right),
\end{align*}
where $\L(X^\alpha,\alpha,Z) = \PP \circ (X^\alpha,\alpha,Z)^{-1}$ and $F : \P(D(\R_+;\R^d) \times M(\R_+;A) \times D(\R_+;\R^m)) \rightarrow \R$, and we recall that $M(\R_+;A)$ denotes and space of (equivalence classes of Lebesgue-a.e.\ equal) measurable functions from $\R_+$ to $A$.
The following theorem shows how to relax the $\FF^Y$-adaptedness requirement:

\begin{theorem} \label{th:control}
Suppose the restriction of $F$ to the set $\{\L(X^\alpha,\alpha,Z) : \alpha \in \AA(\GG)\}$ is continuous.
Then
\begin{align}
\sup_{\alpha \in \AA(\FF^Y)}F(\L(X^\alpha,\alpha,Z)) = \sup_{\alpha \in \AA(\GG)}F(\L(X^\alpha,\alpha,Z)).
\end{align}
\end{theorem}
\begin{proof}
The inequality $(\le)$ is clear, as $\AA(\FF^Y) \subset \AA(\GG)$. To prove the reverse,  it suffices to approximate an arbitrary $\alpha \in \AA(\GG)$ in a suitable sense. To apply Theorem \ref{th:density-measurable}, note that $^*\F^\alpha_t$ is conditionally independent of $Z$ given $\F^Y_t$, for every $t \ge 0$. Hence, there exists a sequence $\alpha^n \in \AA(\FF^Y)$ such that $(\alpha^n,Z) \Rightarrow (\alpha,Z)$ in $M(\R_+;A) \times D(\R_+;\R^d)$. Finally, it should not be surprising that we can conclude that $(X^{\alpha^n},\alpha^n,Z) \Rightarrow (X^\alpha,\alpha,Z)$, e.g., by using the results of Jacod and M\'emin \cite{jacodmemin-weaksolution}, namely Theorems 2.25, 3.24, and 4.5 therein.
\end{proof}

As a special case, if $Y=Z$, then for $\GG$ we may choose any filtration compatible to $\FF^Z$ in the sense of any of the equivalent conditions of Theorem \ref{th:compatible}.
Note also that we could easily generalize Theorem \ref{th:control} by replacing the Lipschitz assumption with some kind of weak uniqueness, and by weakening the continuity assumption on $F$ to some order of Wasserstein-continuity as long as there are suitable moment estimates available for the processes $X^\alpha$.

Traditionally, $F$ is of the form $F(P) = \int G\,dP$, where $G$ is the sum of an integrated running reward and a terminal reward. When $F$ is linear in this sense, an alternative proof of Theorem \ref{th:control} would proceed by approximating a general control $\alpha \in \A(\GG)$ by a piecewise constant one and then using the extreme point representation of Theorem \ref{th:extremepoints-discrete} (cf. \cite[Section 4]{karoui-tanII}). For nonlinear functions $F$, however, the extreme point argument fails, and one must resort to Theorem \ref{th:density-measurable}. Optimal control problems with nonlinear $F$ have become increasingly relevant in recent years in the context of controlled McKean-Vlasov systems (see \cite{carmona2015forward,lacker2017limit,pham2015bellman} and the references therein).

A key use of a result like Theorem \ref{th:control} is in finding compactness to facilitate existence proofs. Theorem \ref{th:compatible} shows that we can conduct the optimization over all \emph{compatible} joint laws of $(\alpha,Z)$, and we have seen that the compatibility condition defines a closed set of probability measures. A final important step is to choose a better path space for the controls, as opposed to $M(\R_+;A)$, in which compact sets are scarce. This is typically done by working with \emph{relaxed controls}, also known as \emph{Young measures}, but we do not go through this here (see \cite{elkaroui-compactification,haussmannlepeltier-existence,kurtzstockbridge-1998}).

\bibliographystyle{amsplain}
\bibliography{bibl}

\end{document}